\documentclass[11pt,a4paper]{article}

\usepackage[utf8]{inputenc}  
\usepackage[T1]{fontenc}  
\usepackage[english]{babel}
\usepackage{tikz}
\usetikzlibrary{positioning,chains,fit,shapes,calc}

\usepackage{geometry}
\usepackage{amsmath}
\usepackage{amsfonts}
\usepackage{amssymb}
\usepackage{amsthm}
\usepackage{stmaryrd}
\usepackage{dsfont}
\usepackage{bbm}
\usepackage{mathrsfs}  

\newtheorem{thm}{Theorem}[section]
\newtheorem{prop}[thm]{Proposition}
\newtheorem{lem}[thm]{Lemma}
\newtheorem{cor}[thm]{Corollary}
\newtheorem{rk}[thm]{Remark}
\newtheorem{ex}[thm]{Example}
\newtheorem*{rk*}{Remark}
\newtheorem{defi}[thm]{Definition}
\newtheorem*{thm*}{Theorem}
\newtheorem*{prop*}{Proposition}
\newtheorem*{lem*}{Lemma}
\usepackage{graphicx}
\usepackage{float}
\restylefloat{figure}

\newcommand{\E}{\mathbb{E}}
\renewcommand{\P}{\mathbb{P}}

\newcommand{\e}{\mathbbm{e}}
\newcommand{\R}{\mathbb{R}}
\newcommand{\N}{\mathbb{N}}
\newcommand{\Z}{\mathbb{Z}}

\newcommand{\cB}{\mathcal{B}}
\newcommand{\cA}{\mathcal{A}}

\newcommand{\cT}{\mathcal{T}}
\newcommand{\D}{\mathbb{D}}

\usepackage[colorlinks=false]{hyperref}

\usepackage{todonotes}

\author{Paul Thevenin\footnote{CMAP \& \'Ecole polytechnique, paul.thevenin@polytechnique.edu \newline The author acknowledges partial support from Agence Nationale de la Recherche,
Grant Number ANR-14-CE25-0014 (ANR GRAAL).}}
\title{Vertices with fixed outdegrees in large Galton--Watson trees}
\date{}
\begin{document}

\maketitle
\begin{abstract}
\small{We are interested in nodes with fixed outdegrees in large conditioned Galton--Watson trees. We first study the scaling limits of processes coding the evolution of the number  of such nodes in different explorations of the tree (lexicographical order and contour visit counting process) starting from the root. We give necessary and sufficient conditions for the limiting processes to be centered, thus measuring the linearity defect of the evolution of the number of nodes with fixed outdegrees. This extends results by Labarbe \& Marckert in the case of the contour visit counting process of leaves in uniform plane trees. Then, we extend results obtained by Janson concerning the asymptotic normality of the number of nodes with fixed outdegrees.}
\end{abstract}

\section{Introduction}

Much attention has been recently given to the fine structure of large random trees. In this paper, we focus particularly on the distribution of vertex degrees in large conditioned Galton--Watson trees, and on how they are spread out in these trees.

\paragraph*{Motivations.} The study of scaling limits of Galton--Watson trees (in short, GW trees) with critical offspring distribution (that is with mean $1$) conditioned by their number of vertices has been initiated by  Aldous \cite{Ald91a, Ald91b, Ald93}. Aldous showed that the scaling limit of large critical GW trees with finite variance is the so-called Brownian continuum random tree (CRT). As a side result, he proved the convergence of their properly rescaled contour functions, which code the trees, to the Brownian excursion. This result was extended by Duquesne, who showed that the scaling limits of critical GW trees, when the offspring distribution has infinite variance and is in the domain of attraction of a stable law, are $\alpha$-stable trees (with $\alpha \in(1,2]$), which were introduced by Le Gall \& Le Jan \cite{LGLJ98} and Duquesne \& Le Gall \cite{DLG02}. From a more discrete point of view,  Abraham and Delmas \cite{AD14b,AD14a} extended the work of Kesten \cite{Kes86} and Janson \cite{Jan12} by describing in full generality the local limits of critical GW trees conditioned to have a fixed large number of vertices.

The number of vertices with a fixed outdegree in large conditioned critical GW trees with finite variance was studied by Kolchin \cite{Kol86}, who showed that it is asymptotically normal. This topic has recently triggered a renewed interest. Minami \cite{Min05} established that these convergences hold jointly under an additional moment condition, which was later lifted by Janson \cite{Jan16}. Rizzolo \cite{Riz15}   considered more generally  GW trees conditioned on a given number of vertices with outdegree in a given set.  One of the motivations for studying these quantities is that there is a variety of random combinatorial models coded by GW trees in which vertex degrees represent a quantity of interest. For example, in \cite{AB15}, vertex degrees code sizes of $2$-connected blocs in random maps and, in \cite{Kor14}, vertex degrees code sizes of faces in dissections.  Also, Labarbe \& Marckert \cite{LM07} studied the evolution of the number of leaves in the contour process of a large uniform plane tree.

\paragraph*{Evolution of vertices with fixed outdegrees.} Our first contribution concerns scaling limits of processes coding the evolution of vertices with fixed outdegrees in different explorations of large GW trees   starting from the root. We shall explore the tree in two ways by using either the contour process (which was considered by Labarbe \& Marckert \cite{LM07}), or the lexicographical order.

In order to state our result, we need to introduce some quick background and notation (see Section \ref{background} for formal definitions). An offspring distribution $\mu$, which is a probability distribution on $\mathbb{Z}_{+}$, is said to be critical if it has mean $1$. To simplify notation, we set $\mu_{i}=\mu(i)$ for $i \geq 0$. If $T$ is a plane tree and $\cA \subset \mathbb{Z}_{+}$, we say that a vertex of $T$ is a $ \cA$-vertex if its outdegree (or number of children) belongs to $ \cA$. 
We define $N^{\cA}(T)$ as the number of $ \cA$-vertices in $T$, and we set $\mu_{\cA} = \sum_{i \in \cA} \mu_i$ to simplify notation. We say that $\cT$ is a $\mu$-GW tree if it is a GW tree with offspring distribution $\mu$. We will always implicitly assume, for the sake of simplicity, that the support of the offspring distribution $\mu$ is non-lattice (a subset $A \subset \Z$ is lattice if there exists $b \in \Z$ and $d \geq 2$ such that $A \subset b + d \Z$), so that for every $n$ sufficiently large a $\mu$-GW tree conditioned on having $n$ vertices is well defined (but all the results carry through to the lattice setting with mild modifications). For $n \geq 1$, we denote by $\cT_n$ a $\mu$-GW tree conditioned to have $n$ vertices.

Let $T$ be a plane tree with $n$ vertices. To define the contour function $(C_{t}(T),0 \leq t \leq 2 n)$ of  $T$, imagine a particle that explores the tree from the left to the right, starting from the root and moving at unit speed along the edges. Then, for $0 \leq t \leq 2(n-1)$, $C_{t}(T)$ is defined as the distance to the root of the position of the particle at time $t$. We set $C_{t}(T)=0$ for $t \in [2(n-1),2n]$ (see Fig.~\ref{fig:exemple} for an example). For every $0 \leq t \leq 1$, let $N_{2nt}^\cA(T)$ be the number of different $\cA$-vertices already visited by  $C(T)$ at time $\lfloor 2nt \rfloor$. In particular, $N_{2n}^\cA(T)=N^\cA(T)$.
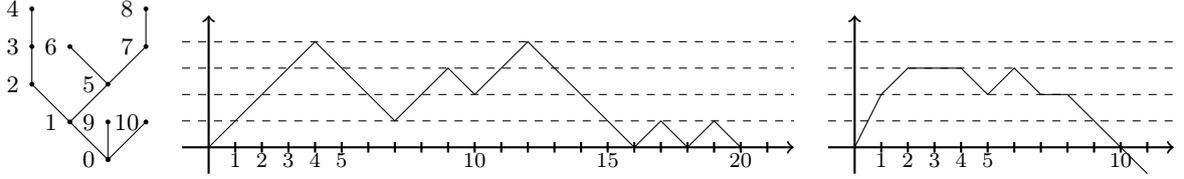
\begin{figure}
\begin{center}
\begin{tabular}{lll}

\begin{tikzpicture}[scale=.5]
\draw (0,1) -- (0,0) -- (1,1) ;
\draw (0,0) -- (-1,1) -- (-2,2) -- (-2,3) -- (-2,4) ;
\draw (-1,1) -- (0,2) -- (-1,3) ;
\draw (0,2) -- (1,3) -- (1,4) ;
 \foreach \Point in {(0,0), (0,1), (1,1), (-1,1), (0,2), (-2,2), (-2,3), (-1,3), (1,3), (1,4), (-2,4), }
        \draw[fill=black] \Point circle (0.05);
\node at (-.5,0) {\footnotesize 0};
\node at (-1.5,1) {\footnotesize 1};
\node at (-.5,1) {\footnotesize 9};
\node at (.5,1) {\footnotesize 10};
\node at (-2.5,2) {\footnotesize 2};
\node at (-.5,2) {\footnotesize 5};
\node at (-2.5,3) {\footnotesize 3};
\node at (-2.5,4) {\footnotesize 4};
\node at (-1.5,3) {\footnotesize 6};
\node at (.5,3) {\footnotesize 7};
\node at (.5,4) {\footnotesize 8};

\end{tikzpicture}
&
\begin{tikzpicture}[scale=.35]
\draw (0,0) -- (1,1) -- (2,2) -- (3,3) -- (4,4) -- (5,3) -- (6,2) -- (7,1) -- (8,2) -- (9,3) -- (10, 2) -- (11,3) -- (12,4) -- (13,3) --(14,2) --(15,1) -- (16,0) -- (17,1) -- (18,0) --(19,1) --(20,0) ;
  \draw[thick,->] (-1,0) -- (22,0);
  \draw[dashed] (-1,1) -- (22,1);
  \draw[dashed] (-1,2) -- (22,2);
  \draw[dashed] (-1,3) -- (22,3);
  \draw[dashed] (-1,4) -- (22,4);
  \draw[thick,->] (0,-1) -- (0,5);
    \draw[thick] (1,.2) -- (1,-.2);
  \draw[thick] (2,.2) -- (2,-.2);
  \draw[thick] (3,.2) -- (3,-.2);
  \draw[thick] (4,.2) -- (4,-.2);
  \draw[thick] (5,.2) -- (5,-.2);
  \draw[thick] (6,.2) -- (6,-.2);
  \draw[thick] (7,.2) -- (7,-.2);
  \draw[thick] (8,.2) -- (8,-.2);
  \draw[thick] (9,.2) -- (9,-.2);
  \draw[thick] (10,.2) -- (10,-.2);
  \draw[thick] (11,.2) -- (11,-.2);
  \draw[thick] (12,.2) -- (12,-.2);
  \draw[thick] (13,.2) -- (13,-.2);
  \draw[thick] (14,.2) -- (14,-.2);
  \draw[thick] (15,.2) -- (15,-.2);
  \draw[thick] (16,.2) -- (16,-.2);
  \draw[thick] (17,.2) -- (17,-.2);
  \draw[thick] (18,.2) -- (18,-.2);
  \draw[thick] (19,.2) -- (19,-.2);
  \draw[thick] (20,.2) -- (20,-.2);
  \draw[thick] (21,.2) -- (21,-.2);
  \node at (1,-.5){\scriptsize 1};
  \node at (2,-.5){\scriptsize 2};
  \node at (3,-.5){\scriptsize 3};
  \node at (4,-.5){\scriptsize 4};
  \node at (5,-.5){\scriptsize 5};
  \node at (10,-.5){\scriptsize 10};
  \node at (15,-.5){\scriptsize 15};
  \node at (20,-.5){\scriptsize 20};
\end{tikzpicture}
&
\begin{tikzpicture}[scale=.35]
  \draw (0,0) -- (1,2) -- (2,3) -- (3,3) -- (4,3) -- (5,2) -- (6,3) -- (7,2) -- (8,2) -- (9,1) -- (10,0) -- (11, -1) ;
  \draw[thick,->] (-1,0) -- (12,0);
  \draw[thick,->] (0,-1) -- (0,5);
  \draw[dashed] (-1,1) -- (12,1);
  \draw[dashed] (-1,2) -- (12,2);
  \draw[dashed] (-1,3) -- (12,3);
  \draw[dashed] (-1,4) -- (12,4);
  \draw[thick] (1,.2) -- (1,-.2);
  \draw[thick] (2,.2) -- (2,-.2);
  \draw[thick] (3,.2) -- (3,-.2);
  \draw[thick] (4,.2) -- (4,-.2);
  \draw[thick] (5,.2) -- (5,-.2);
  \draw[thick] (6,.2) -- (6,-.2);
  \draw[thick] (7,.2) -- (7,-.2);
  \draw[thick] (8,.2) -- (8,-.2);
  \draw[thick] (9,.2) -- (9,-.2);
  \draw[thick] (10,.2) -- (10,-.2);
  \draw[thick] (11,.2) -- (11,-.2);
  \node at (1,-.5){\scriptsize 1};
  \node at (2,-.5){\scriptsize 2};
  \node at (3,-.5){\scriptsize 3};
  \node at (4,-.5){\scriptsize 4};
  \node at (5,-.5){\scriptsize 5};
  \node at (10,-.5){\scriptsize 10};
\end{tikzpicture}
\end{tabular}
\caption{From left to right: a plane tree $T$ with its vertices listed in the depth-first search order, its contour function $C(T)$ and a linear interpolation of its Lukasiewicz path $W(T)$.}
\label{fig:exemple}
\end{center}
\end{figure}

When $\mu$ follows a geometric distribution of parameter $1/2$ (so that $\cT_{n}$ follows the uniform distribution on the set of all plane trees with $n$ vertices) and $ \cA= \{0\}$, Labarbe \& Marckert showed that the convergence
\begin{align*}
\left(  \frac{C_{2nt}(\mathcal{T}_n)}{\sqrt{n}}  ,  \frac{N^{ \{0\} }_{2nt}(\mathcal{T}_n) - nt \mu_0}{\sqrt{n}}  \right)_{0 \leq t \leq 1} 
 \quad \mathop{\longrightarrow}^{(d)}_{n \rightarrow \infty} \quad 
\left( \sqrt{2} \mathbbm{e}_t, B_{t} \right) _{0 \leq t \leq 1}
\end{align*}
holds jointly in distribution in $ \mathcal{C}([0,1],\R^{2})$, where $\mathbbm{e}$ is the normalized Brownian excursion, $B$ is a Brownian motion independent of $\mathbbm{e}$ and $ \mathcal{C}([0,1],\R^{2})$ is the space of continuous $\R^{2}$-valued functions on $[0,1]$ equipped with the uniform topology. 

In words,  the counting process $N^{ \{0\} }(\mathcal{T}_n)$ behaves linearly at the first order, and has centered Brownian fluctuations. Labarbe and Marckert themselves highlight (just after Theorem $4$ in \cite{LM07}) the fact that the fluctuations are centered and do not depend on the final shape of the contour function of the tree, which is quite puzzling. It is therefore natural to wonder if such fluctuations are universal: what happens if the tree is not uniform, if one considers different outdegrees, or if the underlying exploration process is different?

Before stating our result in this direction, we define the second exploration we shall use. If $T$ is a plane tree with $n$ vertices, we denote by $(v_{i}(T))_{0 \leq i \leq n-1}$ the vertices of $T$ ordered in the lexicographical order (also known as the depth-first order).
The Lukasiewicz path $(W_{i}(T))_{0 \leq i \leq n}$ of $T$ is defined by $W_{0}(T)=0$ and $W_{i}(T)-W_{i-1}(T)=k_{v_{i-1}}(T)-1$ for $1 \leq i \leq n$, where $k_{v_{i}}(T)$ denotes the outdegree of $v_{i}$ (see Fig. \ref{fig:exemple} for an example). For $t \in [0,n]$, we set $W_t(T)=W_{\lfloor t \rfloor}(T)$.
For $t \in [0,1]$, we define $K_{nt}^\cA(T)$ as the number of $\cA$-vertices visited by $W(T)$ at time $\lfloor nt \rfloor$ (in other words, $K_{nt}^\cA(T)$ is the number of $\cA$-vertices in the first $\lfloor nt \rfloor$ vertices of $T$ in the lexicographical order). In the next result, convergences hold in distribution in the space $\D([0,1],\R^{2})$ of càdlàg processes on $[0,1]$ equipped with the Skorokhod $J_{1}$ topology (for technical reasons it is  simpler to work with càdlàg processes; see \cite[Chap. VI]{JS03} for background).

\begin{thm}
\label{thm:processus}
Let $\mu$ be a critical distribution with finite variance $\sigma^2>0$ and $\mathcal{T}_n$ be a $\mu$-GW tree conditioned to have exactly $n$ vertices. Let $\cA  \subset \mathbb{Z}_{+}$ be such that $\mu_{\cA}>0$, and set $\gamma_{\cA}= \sqrt{\mu_{\cA} (1-\mu_{\cA}) - \frac{1}{\sigma^2} \left(\sum_{i \in \cA}(i-1)\mu_{i}  \right)^2}$. Then the following assertions hold:
\begin{itemize}
\item[(i)] We have
\begin{align*}
\left( \frac{W_{nt}(\mathcal{T}_n)}{\sqrt{n}} , \frac{K^{\cA}_{nt}(\cT_n)-nt \mu_\cA}{\sqrt{n}} \right) _{0 \leq t \leq 1}  \quad \mathop{\longrightarrow}^{(d)}_{n \rightarrow \infty} \quad 
\left( \sigma \mathbbm{e}_t, \frac{\sum_{i \in \cA} (i-1) \mu_i}{\sigma} \mathbbm{e}_t + \gamma_{\cA} B_t \right) _{0 \leq t \leq 1}
\end{align*}
where $B$ is a standard Brownian motion independent of $\mathbbm{e}$ (see Fig. \ref{fig:exemple2} for a simulation).

\item[(ii)] The following convergence holds in distribution, jointly with that of (i):
\begin{align*}
\left(  \frac{C_{2nt}(\mathcal{T}_n)}{\sqrt{n}}  ,  \frac{N^{\cA}_{2nt}(\mathcal{T}_n) - nt \mu_\cA}{\sqrt{n}}  \right)_{0 \leq t \leq 1} 
 \quad \mathop{\longrightarrow}^{(d)}_{n \rightarrow \infty} \quad 
\left( \frac{2}{\sigma} \mathbbm{e}_t, \frac{\sum_{i \in \cA} i \mu_i}{\sigma} \mathbbm{e}_t + \gamma_\cA {B}_t \right) _{0 \leq t \leq 1}.
\end{align*}
\end{itemize}
\end{thm}

\begin{figure}
\includegraphics[scale=.75]{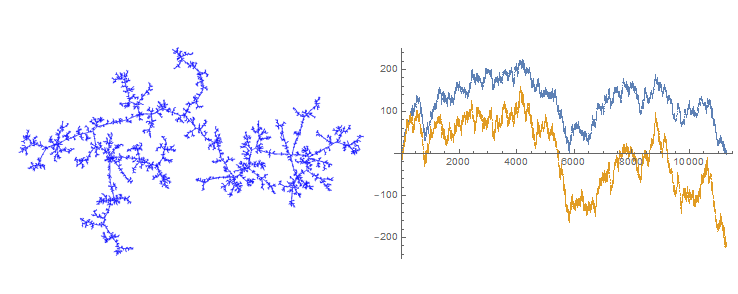}
\caption{ A simulation of a Poisson($1$)-GW tree $\mathcal{T}_n$ with $n=11500$ vertices. Left: an embedding of $\mathcal{T}_n$ in the plane. Right: its Lukasiewicz path together with its renormalized number of $\{1\}$-vertices $( {W_{nt}(\mathcal{T}_n)}/{\sqrt{n}} , {(K^{ \{1\} }_{nt}(\cT_n)-nt \mu_{\{1\}})}/{\sqrt{n}} )_{0 \leq t \leq 1}$. The second one evolves asymptotically as half of the first one plus an independent Brownian motion. }
\label{fig:exemple2}
\end{figure}

As was previously mentioned, assertion (ii) of Theorem \ref{thm:processus}, in the particular case where $ \mathcal{A}= \{0\}$ and $\mu$ is a geometric $1/2$ offspring distribution, was proved by Labarbe \& Marckert \cite{LM07}. It turns out that for leaves, the fluctuations of the counting process $N^{ \{0\} }({\mathcal{T}_n})$ are always centered, irrespective of the offspring distribution. However, the fluctuations are different when one considers other outdegrees or the lexicographical order instead of the contour visit counting process.

Let us briefly comment on the strategy of the proof of Theorem \ref{thm:processus}, which is different from the approach of Labarbe \& Marckert (who rely on explicit formulas for the number of paths with $\pm 1$ steps and various constraints). We start by working with the Lukasiewicz path and establish Theorem \ref{thm:processus} (i) by combining  a general formula giving the joint distribution of outdegrees in GW trees in terms of random walks (Section \ref{sec:jointdistribution}) with absolute continuity arguments and the Vervaat transform.  Theorem \ref{thm:processus} (ii) is then a rather direct consequence of (i) by relating the contour exploration to the depth-first search exploration (see in particular Lemma \ref{lem:relationbC}).

In Section \ref{sec:secgen}, we extend Theorem \ref{thm:processus} (ii)  when we only take into account the $k$-th time we visit a vertex with outdegree $i$ (with $k,i$   integers such that $1 \leq k \leq i+1$). To this end, we give a description of the structure of branches in the tree using  binomial-tail inequalities, which could be of independent interest.

Finally, an extension of this theorem to offspring distributions with infinite variance can be found in Section \ref{ss:ext}.

\paragraph*{Asymptotic normality of the number of vertices with fixed outdegree.} Our next contribution is to extend the joint asymptotic normality of the number of vertices with a fixed outdegree in large conditioned critical GW trees  obtained by Janson \cite{Jan16}, by counting vertices whose outdegree belongs to a fixed subset of $\mathbb{Z}_{+}$ and by allowing a more general conditioning. Indeed, we shall focus on  $\mu$-GW trees conditioned to have $n$ $\cB$-vertices, for a fixed $\cB \subset \Z_+$ (we shall always implicitly restrict ourselves to values of $n$ such that this conditioning makes sense).

\begin{thm}
\label{thm:normality}
Let $\mu$ be a critical offspring distribution with positive finite variance and let $\cA, \cB $ be subsets of $\mathbb{Z}_{+}$ such that $\mu_\cB>0$. For $n \geq 1$, let $\mathcal{T}_n^{\cB}$ be a $\mu$-GW tree conditioned to have $n$ $\cB$-vertices. Then:
\begin{enumerate}
\item[(i)]  as $n \rightarrow \infty$, $\frac{1}{n}\mathbb{E}(N^{\cA}(\mathcal{T}_n^\cB)) \rightarrow \frac{\mu_{\cA}}{\mu_{\cB}}$;
\item[(ii)] there exists $\delta_{\cA,\cB} \geq 0$ such that the convergence
\begin{equation}
\frac{N^{\cA}(\mathcal{T}_n^{\cB})-n \frac{\mu_{\cA}}{\mu_{\cB}}}{\sqrt{n}}  \quad \mathop{\longrightarrow}^{(d)}_{n \rightarrow \infty} \quad  \mathcal{N}(0,\delta_{\cA,\cB}^2)\label{eq:normalite}
\end{equation}
holds in distribution, where $\mathcal{N}(0,\delta_{\cA,\cB}^2)$ is a centered Gaussian random variable with variance $\delta_{\cA,\cB}^2$. In addition, $\delta_{\cA,\cB}=0$ if and only if $\mu_\cA=0$ or $\mu_{\cA \backslash \cB} = \mu_{\cB \backslash \cA} = 0$.
\item[(iii)] the convergences \eqref{eq:normalite} hold jointly for $ \mathcal{A} \subset \Z_{+}$, in the sense that for every $j \geq 1$ and $\cA_1, \cdots, \cA_j \subset \Z_{+}$, $ (({N^{\cA_i}(\mathcal{T}_n^{\cB})-n \frac{\mu_{\cA_i}}{\mu_{\cB}}})/{\sqrt{n}})_{1 \leq i \leq j} $ converges in distribution to a Gaussian vector.
\end{enumerate}
\end{thm}

As previously mentioned, this extends results of Kolchin \cite{Kol86}, Minami \cite{Min05} and  Janson \cite{Jan16}. The main idea is, roughly speaking, to use a general formula giving the joint distribution of outdegrees in GW trees in terms of random walks of Section \ref{sec:jointdistribution} (which was already used in the proof of Theorem \ref{thm:processus}), combined with various local limit estimates (Section \ref{sec:asymptotic_normality}). As we will see (cf \eqref{eq:NZ}), in the case $ \cA=\mathbb{Z}_{+}$, we have $\delta_{\cA,\cB}^{2}={\gamma^2_{\cB}}/{\mu^3_{\cB}}$ (with $\gamma_{\cB}$ defined as in Theorem \ref{thm:processus} by replacing $\cA$ by $\cB$). Also, the proof of Theorem \ref{thm:normality} (ii) gives a way to compute explicitly $\delta_{\cA,\cB}$ (see Example \ref{ex:example} for the explicit values of the variances and covariances in the cases $\cB=\mathbb{Z}_+$ {and $\cB=\{a\}$ for some $a \in \Z_+$}). See Section \ref{ss:ext} for discussions concerning other offspring distributions.

Our approach, based on a multivariate local limit theorem, applies more generally when $\mu$ is in the domain of attraction of a stable law. In this case, it allows us to prove the convergence of $\cT_n^\cB$ (properly renormalized) towards a Lévy tree, thus generalizing \cite[Theorem $8.1$]{Kor12} which was stated only under the condition that $\cB$ or $\Z_+ \backslash \cB$ is finite. These new results can be found in Section \ref{ss:ext}.

\paragraph*{Acknowledgements.}

I would like to thank Igor Kortchemski for suggesting the study of these exploration processes, as well as for his precious and stimulating comments on the multiple versions of this paper. I would also like to thank the reviewers of this paper, for their numerous constructive remarks and corrections, leading in particular to shorter proofs of Lemma \ref{lem:asymp} (i) and Proposition \ref{prop:ancestors}.

\section{Background on trees and their codings}
\label{background}

We start by recalling some definitions and useful well-known results concerning Galton-Watson trees and their coding by random walks (we refer to \cite{LG05} for details and proofs).

\paragraph{Plane trees.} We first define plane trees using Neveu's formalism \cite{Nev86}. First, let $\N^* = \left\{ 1, 2, \ldots \right\}$ be the set of all positive integers, and $\mathcal{U} = \cup_{n \geq 0} (\N^*)^n$ be the set of finite sequences of positive integers, with $(\N^*)^0 = \{ \emptyset \}$ by convention.
By a slight abuse of notation, for $k \in \Z_+$, we write an element $u$ of $(\N^*)^k$ by $u=u_1 \cdots u_k$, with $u_1, \ldots, u_k \in \N^*$. For $k \in \Z_+$, $u=u_1\cdots u_k \in (\N^*)^k$ and $i \in \Z_+$, we denote by $ui$ the element $u_1 \cdots u_ki \in (\N^*)^{k+1}$ and $iu$ the element $iu_1 \cdots u_k \in (\N^*)^{k+1}$. A tree $T$ is a subset of $\mathcal{U}$ satisfying the following three conditions:
(i) $\emptyset \in T$ (the tree has a root); (ii) if $u=u_1\cdots u_n \in T$, then, for all $k \leq n$, $u_1\cdots u_k \in T$ (these elements are called ancestors of $u$); (iii) for any $u \in T$, there exists a nonnegative integer $k_u(T)$ such that, for every $i \in \N^*$, $ui \in T$ if and only if $1 \leq i \leq k_u(T)$ ($k_u(T)$ will be called the number of children of $u$, or the outdegree of $u$). The elements of $T$ are called the vertices of $T$. The set of all the ancestors of a vertex $u$ will be called the ancestral line of $u$, by analogy with genealogical trees. We denote by $|T|$ the total number of vertices of $T$.

The \textit{lexicographical order} $\prec$ on $\mathcal{U}$ is defined as follows:  $\emptyset \prec u$ for all $u \in \mathcal{U} \backslash \{\emptyset\}$, and for $u,w \neq \emptyset$, if $u=u_1u'$ and $w=w_1w'$ with $u_1, w_1 \in \N^*$, then we write $u \prec w$ if and only if $u_1 < w_1$, or $u_1=w_1$ and $u' \prec w'$.  The lexicographical order on the vertices of a tree $T$ is the restriction of the lexicographical order on $\mathcal{U}$; for every $0 \leq k \leq |T|-1$ we write $v_k(T)$, or $v_k$ when there is no confusion, for the $(k+1)$-th vertex of $T$ in the lexicographical order. Recall from the introduction that the Lukasiewicz path $(W_{i}(T))_{0 \leq i \leq |T|}$ of $T$ is defined by $W_{0}(T)=0$ and $W_{i}(T)-W_{i-1}(T)=k_{v_{i-1}}(T)-1$ for $1 \leq i \leq |T|$.

\paragraph{Galton--Watson trees.} Let $\mu$ be an offspring distribution with mean at most $1$ such that $\mu(0)+\mu(1)<1$ (implicitly, we always make this assumption to avoid degenerate cases). A GW tree $ \mathcal{T}$ with offspring distribution $\mu$ (also called $\mu$-GW tree) is a random variable taking values in the space of all finite plane trees, characterized by the fact that $ \P( \mathcal{T}=T)= \prod_{u \in T} \mu_{k_{u}(T)}$ for every finite plane tree $T$.  We also always implicitly assume that $\gcd(i \in \Z_+, \mu_i > 0) =1$, so that $\P( |\mathcal{T}|=n)>0$ for every $n$ sufficiently large ($\mu$ is said to be aperiodic). All the results can be adapted to the periodic setting with mild modifications.

A key tool to study GW trees is the fact that their Lukasiewicz path is, roughly speaking, a killed random walk, which allows to obtain information on GW trees from the study of random walks. More precisely, let $S$ be the random walk on $\Z_{+} \cup \{-1\}$ starting from $S_0=0$ with jump distribution given by $\P(S_1=i)=\mu_{i+1}$ for $i \geq -1$ (we keep the dependency of $S$ in $\mu$ implicit). The proof of the following lemma can be found in \cite{LG05}.

\begin{lem} 
\label{lem:codage}
 Let $\mu$ be an offspring distribution with mean at most $1$ and $\cT_{n}$ be a $\mu$-GW tree conditioned on having $n$ vertices. Then $(W_{i}(\cT_{n}))_{0 \leq i \leq n}$ has the same distribution as $(S_i)_{0 \leq i \leq n}$ conditionally given the event $\{S_n=-1,\ \forall \ 0 \leq i \leq n-1, S_i \geq 0\}$.
\end{lem}

\paragraph{Several useful ingredients.} We finally gather two very useful ingredients. The first one is a joint scaled convergence in distribution of the contour process (which was defined in the introduction) and the Lukasiewicz path of a critical GW tree with finite variance, conditioned to have $n$ vertices, to the same Brownian excursion.

\begin{thm}[Marckert and Mokkadem \cite{MM03}, Duquesne \cite{Duq03}]
\label{thm:CW}
Let $\mu$ be a critical offspring distribution with finite positive variance $\sigma^2$. Then  the following convergence holds jointly in distribution:  
\begin{align*}
\left( \frac{C_{2nt}(\mathcal{T}_n)}{\sqrt{n}}, \frac{W_{nt}(\cT_{n})}{\sqrt{n}} \right)_{0 \leq t \leq 1} \overset{d}{\longrightarrow} 
\left( \frac{2}{\sigma} \mathbbm{e}_t, \sigma \mathbbm{e}_t \right)_{0 \leq t \leq 1}
\end{align*}
where $\mathbbm{e}$ has the law of the normalized Brownian excursion.
\end{thm}
This result is due to Marckert and Mokkadem \cite{MM03} under the assumption that $\mu$ has a finite exponential moment. The result in the general case can be deduced from \cite{Duq03}, however it is not clearly stated in this form. See \cite[Theorem $8.1$, (II)]{Kor12} (taking $\mathcal{A}=\Z_+$ in this theorem) for a precise statement.

The second ingredient is the local limit theorem (see \cite[Theorem 4.2.1]{IL71}).

\begin{thm}
\label{llt}
Let $(S_{n})_{n \geq 0}$ be a random walk on $\Z$ such that the law of $S_1$ has finite positive variance $\sigma^2$. Let $h \in \Z_+$ be the maximal integer such that there exists $b \in \Z$ for which $Supp(S_{1}) \subset b + h \Z$. Then, for such $b \in \Z$,
\begin{align*}
\underset{k \in \mathbb{Z}}{\sup} \left| \sqrt{2 \pi \sigma^2 n} \mathbb{P} (S_n=nb+kh) - h \exp \left( -\frac{1}{2} \left( \frac{nb+kh- n \E(S_{1})}{\sigma \sqrt{n}} \right)^2 \right) \right|  \quad \mathop{\longrightarrow}_{n \rightarrow \infty} \quad  0.
\end{align*}
\end{thm}
When $Supp(S_{1})$ is non-lattice,  observe that one can take $b=0$ and $h=1$ in the previous result.

This theorem admits the following generalization in the multivariate setting (see e.g. \cite[Theorem $6.1$]{Rva61}). In the multivariate case in dimension $j \geq 1$, we say that a random variable $\mathbf{Y} \in \Z^j$ is aperiodic if there is no strict sublattice of $\Z^j$ containing the set of differences $\{ \textbf{x}-\textbf{y}, \textbf{x}, \textbf{y} \in \Z^j, \P(\textbf{Y}=\textbf{x})>0, \P(\textbf{Y}=\textbf{y})>0\}$. Furthermore, $\mathcal{S}_j$ denotes the set of symmetric positive definite matrices of dimension $j$. 

\begin{thm}
\label{thm:multivariatellt}
Let $j \geq 1$ and $(\textbf{Y}_i)_{i \geq 1} := ((Y_i^{(1)}, \ldots, Y_i^{(j)}))_{i \geq 1}$ be i.i.d. random variables in $\Z^j$, such that the covariance matrix $\Sigma$ of  $\mathbf{Y}_1$ is positive definite. Assume in addition that $\textbf{Y}_1$ is aperiodic, and denote by $\textbf{M}$ the mean of $\textbf{Y}_1$. Finally, define for $n \geq 1$
\begin{align*}
\mathbf{T_n} = \frac{1}{\sqrt{n}} \left(\sum_{i=1}^n \mathbf{Y}_i - n \mathbf{M} \right) \in \R^j.
\end{align*}
Then, as $n \rightarrow \infty$, uniformly for $\bf{x} \in \R^j$ such that $\P \left( \mathbf{T_n} = \mathbf{x} \right) > 0$,
\begin{align*}
\P \left( \mathbf{T_n} = \mathbf{x} \right) = \frac{1}{(2\pi n)^{j/2} \sqrt{\det \Sigma}} e^{-\frac{1}{2} ^t \bf{x} \Sigma^{-1} \bf{x}} + o\left( n^{-j/2} \right).
\end{align*}
\end{thm}

This theorem can easily be adapted when $\mathbf{Y}_1$ is not aperiodic. However, for convenience, we shall restrict ourselves to this case in what follows.

\section{Joint distribution of outdegrees in GW trees}
\label{sec:jointdistribution}

The first steps of the proofs of Theorems \ref{thm:processus} and \ref{thm:normality} both reformulate events on trees in terms of events on random walks, whose probabilities are easier to estimate. In this direction, in this section, we give a general formula for the joint distribution of outdegrees in GW trees in terms of random walks (Proposition \ref{prop:joint}) and establish technical estimates (Lemma \ref{lem:asymp}) which will be later used several times.

\subsection{A joint distribution}
\label{ssec:jointdistribution}

The following proposition is a key tool in the study of the outdegrees in a $\mu$-GW tree $\cT$, as it allows to study the joint distribution of $(N^{\Z_+}(\cT), N^{\cB}(\cT))$:

\begin{prop}
\label{prop:joint}
Let $\cB \subset \Z_+$. Let $(S_i)_{i \geq 0}$ be a random walk starting from $0$, whose jumps are independent and distributed according to $\mu(\cdot + 1)$, and let $(J_i^\cB)_{i \geq 0}$ be the walk starting from $0$ such that, for all $i \geq 0$, $J_{i+1}^\cB-J_i^\cB=\mathds{1}_{S_{i+1}-S_i+1 \in \cB}$. Then, for every $n \geq 1$ and $k \geq 0$,
\begin{align*}
\mathbb{P} \left( N^{\Z_+}(\cT)=n, N^{\cB}(\cT)=k \right)= \frac{1}{n} \mathbb{P}\left(S_n=-1,J_n^\cB=k\right).
\end{align*}
\end{prop}

This proposition is a consequence of the so-called cyclic lemma, which is responsible for the factor $1/n$ (see e.g. \cite[Equation (2)]{Kor12}): roughly speaking, let $(S_i)_{0 \leq i \leq n}$ be a walk starting from $0$ and reaching $-1$ at time $n$. Then, among all $n$ cyclic shifts of $(S_i)_{0 \leq i \leq n}$, exactly one of them takes only nonnegative values between times $0$ and $n-1$.

The following asymptotics, which can be derived from a local limit theorem (see e.~g.~\cite{Riz15} or \cite[Theorem 8.1]{Kor12})  will be useful throughout the paper:
\begin{equation}
\mathbb{P}(N^{\cB}(\cT)=k) \underset{k \rightarrow \infty}{\sim} \frac{1}{\sqrt{2 \pi \sigma^2}} \sqrt{\mu_{\cB}} k^{-3/2},\label{eq:NB}
\end{equation}
assuming that  $\mathbb{P}(N^{\cB}(\cT)=k) >0$ for $k$ sufficiently large.

\subsection{A technical estimate}
\label{ssec:norm}

We introduce other probability measures as follows:

\begin{defi}
Let $ \mathcal{C} \subset \mathbb{Z}_{+}$ be a subset such that $\mu_{\mathcal{C}}>0$. For $i \in \mathbb{Z}$, we set
\begin{align*}
  p_{\mathcal{C}}(i)=
  \begin{cases}
    \frac{\mu_{i+1}}{\mu_{\mathcal{C}}}  & \text{if }  i+1 \in \mathcal{C}\\
    0 & \text{otherwise}. \\
  \end{cases}
\end{align*}
We let $m_{\mathcal{C}}$ be the expectation of $p_{\mathcal{C}}$ and $\sigma^{2}_{\mathcal{C}}$ be its variance.
\end{defi}

The following identities will be useful.

\begin{lem}
\label{relation}
Assume that $\mu$ is critical and has finite positive variance $\sigma^{2}$. Let $\cB \subset \Z_+$ be such that $\mu_\cB>0$ and $\mu_{\cB^c} > 0$. Let $\gamma_{\cB} \geq 0$ be such that $\gamma_{\cB}^{2}= \mu_{\cB} (1-\mu_{\cB}) - \frac{1}{\sigma^2} \left(\sum_{i \in \cB}(i-1)\mu_{i}  \right)^2$.
Then the following identities hold:
\begin{itemize}
\item[(i)] $m_{\cB^c} (1-\mu_{\cB}) + m_{\cB} \mu_{\cB} = 0$,
\item[(ii)] $\gamma_\cB^2 =\mu_\cB (1-\mu_\cB) - \frac{1}{\sigma^2} \mu_\cB^2 m_\cB^2$,
\item[(iii)] $\mu_\cB \sigma_\cB^2 + (1-\mu_\cB) \sigma_{\cB^c}^2 = \frac{\sigma^{2} }{\mu_\cB (1-\mu_\cB)} \gamma_\cB^2$.
\end{itemize}
\end{lem}
In particular, observe that $\gamma_{B}$ is well-defined by (iii). Furthermore, if $\# Supp(\mu) \geq 3$, then at least one of the variances $\sigma^2_{\cB}$ and $\sigma^2_{\cB^c}$ is positive, which implies by (iii) that $\gamma_\cB>0$.

\begin{proof}
For (i), simply write that the quantity $m_{\cB^c} (1-\mu_{\cB}) + m_{\cB} \mu_{\cB}$ is equal to
\begin{align*}
 (1-\mu_{\cB}) \sum_{i \geq -1} i p_{\cB^c}(i) + \mu_{\cB} \sum_{i \geq -1} i p_{\cB}(i)&= (1-\mu_{\cB}) \sum_{i+1 \notin \cB} \frac{i \mu_{i+1}}{1-\mu_{\cB}} + \mu_{\cB} \sum_{i+1 \in \cB} \frac{i \mu_{i+1}}{\mu_{\cB}}\\
&= \sum_{i \notin \cB} (i-1) \mu_i + \sum_{i \in \cB} (i-1) \mu_i,
\end{align*}
which is equal to $0$ since $\mu$ is critical.   The second assertion is clear, while the proof of the last one is similar to the first one and is left to the reader.
\end{proof}

Let us keep the notation of Proposition \ref{prop:joint}. In particular, recall that the walk $(J_i^\cB)_{i \geq 0}$ is defined from $(S_i)_{i \geq 0}$ as $J_0^\cB=0$ and, for $i \geq 0$, $J_{i+1}^\cB-J_i^\cB=\mathds{1}_{S_{i+1}-S_i+1 \in \cB}$. 

We set, for $c \in \R$,
\[k_{n}(c)=\lfloor \mu_{\cB} n + c \sqrt{n} \rfloor.\]

The following estimate will play an important role.

\begin{lem}
\label{lem:asymp}
Let $\mu$ be an aperiodic critical offspring distribution with positive finite variance $\sigma^{2}$ such that $\# Supp(\mu) \geq 3$, and let $\cB \subset \Z_+$ be such that $\mu_\cB>0$ and $\mu_{\cB^c} > 0$. Assume in addition that $p_\cB$ or $p_{\cB^c}$ is aperiodic. Fix $a \in \R$ and let $(a_{n})$ be a sequence of integers such that $a_{n}/\sqrt{n} \underset{n \rightarrow \infty}{\rightarrow} a$.
Then the following assertions hold as $n \rightarrow \infty$, uniformly for $c$ in a compact subset of $\R$:
\begin{itemize}
\item[(i)] $\displaystyle \P\left( S_n=a_n, J_n^\cB=k_n(c) \right) \sim \frac{1}{n} \frac{1}{{2 \pi \sigma  \gamma_\cB}}  e^{  - \frac{1}{2 \sigma^{2}} a^{2} - \frac{1}{2\gamma_{\cB}^{2} } \left( c- \frac{\mu_{\cB} m_{\cB}}{\sigma^{2}}  a\right)^{2}}$,
\item[(ii)] $\P \left( N^{\Z_+}(\cT)=n, N^\cB(\cT)=k_n(c) \right) \sim \displaystyle\frac{1}{n^2} \, \frac{1}{2\pi \sigma \gamma_\cB} e^{- \frac{c^2}{2 \gamma_\cB^2}}$.
\end{itemize}
\end{lem}

Observe that (ii) is a straightforward consequence of (i) and Proposition \ref{prop:joint}. (i) itself follows from the multivariate local limit theorem \ref{thm:multivariatellt}:

\begin{proof}[Proof of Lemma \ref{lem:asymp} (i)]

The idea is to apply Theorem \ref{thm:multivariatellt} to a sequence of i.i.d. variables in $Z^2$, distributed as $\mathbf{Y}_1 := (S_1, J_1^\cB)$. Since $p_\cB$ or $p_{\cB^c}$ is aperiodic, $\mathbf{Y}_1$ (as a $2$-dimensional variable) is aperiodic as well. Furthermore, the mean and the covariance matrix of $\mathbf{Y}_1$ are respectively equal to:
\begin{align*}
M= \begin{pmatrix}
0 \\ \mu_\cB
\end{pmatrix} \qquad \text{ and } \qquad 
\Sigma = \begin{pmatrix}
\sigma^2 & \mu_\cB m_\cB \\
\mu_\cB m_\cB & \mu_\cB(1-\mu_\cB).
\end{pmatrix}
\end{align*}
where $\sigma^2$ is the variance of $\mu$. In particular, $\det \Sigma = \sigma^2 \gamma_\cB^2 > 0$.

On the other hand, as $\mu$ is non-lattice, for $n$ large enough, uniformly for $c$ in a compact subset of $\R$, $\P( S_n=a_n, J_n^\cB=k_n(c))>0$.
An easy computation, with the help of Lemma \ref{relation} (ii), gives the result that we want.
\end{proof}

\section{Evolution of outdegrees in an exploration of a Galton-Watson tree}
\label{evo}

The aim of this section is to establish Theorem \ref{thm:processus}. Recall from the introduction that if $T$ is a tree and $\cA \subset \Z_{+}$, $C(T)$ denotes the contour function of $T$, for $0 \leq t \leq 1$, $N_{2nt}^\cA(T)$ denotes the number of different $\cA$-vertices already visited by  $C(T)$ at time $\lfloor 2nt \rfloor$ and $K_{nt}^\cA(T)$ denotes the number of $\cA$-vertices in the first $\lfloor nt \rfloor$ vertices of $T$ in the depth-first search (or, equivalently, the lexicographical order).

We assume here that $\cA \subset \mathbb{Z}_{+}$ is such that $\mu_\cA > 0$. We keep the notation of Section \ref{ssec:jointdistribution}, and denote in particular by $m_{\cA}$ the expectation of a random variable with law given by $p_{\cA}(i)= \frac{\mu_{i+1}}{\mu_{\cA}}\mathds{1}_{i+1\in \cA}$ for $i \in \Z$.

\subsection{Depth-first exploration}
\label{proc1}

In this section, we study the evolution of the number of $\cA$-vertices in conditioned GW trees for the depth-first search, and establish in particular Theorem \ref{thm:processus} (i). Throughout this section, we fix a critical distribution $\mu$ with finite positive variance $\sigma^2$, and we let $\mathcal{T}_n$ denote a $\mu$-GW tree conditioned on having $n$ vertices.

The idea of the proof of Theorem \ref{thm:processus} (i) is the following.  By Lemma \ref{lem:codage}, the convergence of  Theorem \ref{thm:processus} (i)  can be restated  in terms of the random walk
$(S_i)_{0 \leq i \leq n}$ (with jump distribution given by $\P(S_1=i)=\mu(i+1)$ for $i \geq -1$) conditionally given the event $\{S_n=-1,\ \forall \ 0 \leq i \leq n-1, S_i \geq 0\}$. We first establish a result for the ``bridge'' version where one works conditionally given the event $\{S_n=-1\}$ (Lemma \ref{lem:pont}) and then conclude by using the so-called Vervaat transform.

To simplify notation, for every $t \geq 0$, we set $S_{t}=S_{\lfloor t \rfloor }$ and ${J}^{\cA}_{t}= \sum_{k=1}^{\lfloor t \rfloor } \mathds{1}_{ \{S_{k}-S_{k-1} + 1 \in \mathcal{A}\} }$

\begin{lem}
\label{lem:pont}
The following convergence holds in distribution
\begin{equation}
\label{eq:pont}
\left( \frac{S_{ nt }}{\sqrt{n}}, \frac{ J_{nt}^{\mathcal {A}} -  \mu_{\cA} nt}{\sqrt{n}}\right)_{0 \leq t \leq 1}  \quad \textrm{under} \quad  \P( \, \cdot \, | S_{n}=-1)  \quad \mathop{\longrightarrow}^{(d)}_{n \rightarrow \infty} \quad  \left( \sigma B^{br}_t, \frac{\mu_{\cA} m_{\cA}}{\sigma} B^{br}_t + \gamma_{\cA} B'_t \right)_{0 \leq t \leq 1}
\end{equation}
where $B^{br}$ is a standard Brownian bridge and $B'$ is a standard Brownian motion independent of $B^{br}$.
\end{lem}

\begin{proof}
We first check that the corresponding nonconditioned statement holds, namely that  the following convergence holds in distribution :
\begin{equation}
\label{eq:noncond}
\left( \frac{S_{ nt }}{\sqrt{n}}, \frac{J_{nt}^{\cA} - \mu_{\cA} nt}{\sqrt{n}} \right)_{0 \leq t \leq 1}   \quad \mathop{\longrightarrow}^{(d)}_{n \rightarrow \infty} \quad  \left( \sigma B_t, \frac{\mu_{\cA} m_{\cA}}{\sigma} B_t + \gamma_{\cA} B'_t \right)_{0 \leq t \leq 1}
\end{equation}
where $B$ is a standard Brownian motion and $B'$ a standard Brownian motion independent of $B$. To this end, by \cite[Theorem 16.14]{Kal02}, it is enough to check that the one-dimensional convergence holds for $t=1$. By Lemma \ref{lem:asymp} (i), uniformly for  $a, b$ in a compact subset of $\R$:
\[\P(S_{n}= \lfloor a \sqrt{n} \rfloor, J^{\cA}_{n}=\lfloor \mu_{\cA} n+b \sqrt{n}\rfloor)  \quad \mathop{\sim}_{n \rightarrow \infty} \quad   \frac{1}{{2 \pi \sigma  \gamma_\cA}} \frac{1}{n}  e^{  - \frac{1}{2 \sigma^{2}} a^{2} - \frac{1}{2\gamma_{\cA}^{2} } \left( b- \frac{\mu_{\cA} m_{\cA}}{\sigma^{2}}  a\right)^{2}}.\]
It is standard (see e.g.~\cite[Theorem 7.8]{Bil68}) that this implies that $({S_{ n }}/{\sqrt{n}}, ({J_{n}^{\cA} - \mu_{\cA} n})/{\sqrt{n}})$ converges in distribution to 
$( \sigma B_1, \frac{\mu_{\cA} m_{\cA}}{\sigma} B_1+ \gamma_{\cA} B'_1 )$, which yields \eqref{eq:noncond}.

We now establish \eqref{eq:pont} by using an absolute continuity argument. We fix $u \in (0,1)$, a bounded continuous functional $F: \mathbb{D}([0,u],\R^{2}) \rightarrow \R$, and to simplify notation set $A_{n}=\E \left[F ( {S_{ nt }}/{\sqrt{n}}, ({ J_{nt}^{\mathcal {A}} -  \mu_{\cA} nt})/{\sqrt{n}})_{0 \leq t \leq u}  | S_{n}=-1 \right] $. Then, setting $\phi_{n}(i)=\P(S_{n}=i)$, we have
\[A_{n}= \E \left[F\left(\left( \frac{S_{ nt }}{\sqrt{n}}, \frac{ J_{nt}^{\mathcal {A}} -  \mu_{\cA} nt}{\sqrt{n}}\right)_{0 \leq t \leq u}\right) \frac{\phi_{n-\lfloor nu \rfloor}(-S_{\lfloor nu \rfloor}-1)}{\phi_{n}(-1)}\right].\]
 An application of the local limit theorem \ref{llt} allows to write as $n \rightarrow \infty$
\[A_{n} = \E \left[F\left(\left( \frac{S_{ nt }}{\sqrt{n}}, \frac{ J_{nt}^{\mathcal {A}} -  \mu_{\cA} nt}{\sqrt{n}}\right)_{0 \leq t \leq u}\right) \frac{q_{1-u}(-S_{\lfloor nu \rfloor}/\sqrt{n})}{q_{1}(0)}\right]+o(1),
\]
where $q_t$ denotes the density of a centered Brownian motion of variance $\sigma^2$ at time $t$. Therefore, by \eqref{eq:noncond}, as $n \rightarrow \infty$,
\begin{eqnarray}
A_{n} & \displaystyle  \mathop{\longrightarrow}_{n \rightarrow \infty}& \E\left[ F \left(  \left( \sigma B_t, \frac{\mu_{\cA} m_{\cA}}{\sigma} B_t + \gamma_{\cA} B'_t  \right)_{0 \leq t \leq u} \right)\frac{q_{1-u}(-B_u)}{q_1(0)} \right] \notag\\
& =&   \displaystyle\E\left[ F \left(  \left( \sigma B^{br}_t, \frac{\mu_{\cA} m_{\cA}}{\sigma} B^{br}_t + \gamma_{\cA} B'_t  \right)_{0 \leq t \leq u} \right) \right],\label{eq:0u}
\end{eqnarray}
where the last identity follows from standard absolute continuity properties of the Brownian bridge (see e.g.~\cite[Chapter XII]{RY99}).

The convergence \eqref{eq:0u} shows in particular that, conditionally given $S_{n}=-1$, the process $( {S_{ nt }}/{\sqrt{n}}, ({ J_{nt}^{\mathcal {A}} -  \mu_{\cA} nt})/{\sqrt{n}})_{0 \leq t \leq 1} $ is tight on $[0,u]$ for every $u \in (0,1)$. To check that it is tight on $[u,1]$, it suffices to check that for $u \in (0,1)$, $( {S_{ n-nt }}/{\sqrt{n}}, ({ J_{n-nt}^{\mathcal {A}} -  \mu_{\cA} n(1-t)})/{\sqrt{n}})_{0 \leq t \leq u}$ is tight  conditionally given $S_{n}=-1$. To this end, notice that by time-reversal  the process $(\widehat{S}_{i},\widehat{J}_{i})_{0 \leq i \leq n}:=(S_{n}-S_{n-i},J^{\cA}_{n}-J^{\cA}_{n-i})_{0 \leq i \leq n}$ has the same distribution as $(S_{i},J_{i})_{0 \leq i \leq n}$ (and this also holds conditionally given $S_{n}=-1$). Then write
\[\left( \frac{S_{ n-nt }}{\sqrt{n}}, \frac{ J_{n-nt}^{\mathcal {A}} -  \mu_{\cA} n(1-t)}{\sqrt{n}}\right)_{0 \leq t \leq u}= \left( \frac{\widehat{S}_{n}-\widehat{S}_{nt}}{\sqrt{n}}, \frac{\widehat{J}^{\cA}_{n}- \mu_{\cA}n}{\sqrt{n}}- \frac{\widehat{J}^{\cA}_{nt}-\mu_{\cA}nt}{\sqrt{n}}  \right)_{0 \leq t \leq u}.\]
Now, by Lemma \ref{lem:asymp} (i) and the local limit theorem, uniformly for $b$ in a compact subset of $\R$, 
$\P(J^{\cA}_{n}=\lfloor \mu_{\cA} n+b \sqrt{n}\rfloor |S_{n}=-1)\sim    \frac{1}{ \sqrt{{2 \pi  \gamma_\cA n}}}   e^{  - {b^{2}}/{2\gamma_{\cA}^{2}}}$ as $n \rightarrow \infty$, which shows that, conditionally given $S_{n}=-1$,  $(J_{n}^{\cA}- \mu_{\cA}n)/\sqrt{n}$ converges in distribution. Hence by  \eqref{eq:0u},  the process  $( {S_{ nt }}/{\sqrt{n}}, ({ J_{nt}^{\mathcal {A}} -  \mu_{\cA} nt})/{\sqrt{n}})_{u \leq t \leq 1} $ is tight on $[u,1]$  conditionally given $S_{n}=-1$.  This allows us to conclude that this process is actually tight on $[0,1]$, and in addition, this identifies the convergence of the finite dimensional marginal distributions.  \end{proof}

In order to deduce Theorem \ref{thm:processus} (i) from the bridge version of Lemma \ref{lem:pont}, we now use the Vervaat transformation, whose definition is recalled here.

Set $\D_0([0,1],\R)= \{ \omega \in \D([0,1],\R) ; \, \omega(0)=0\}$.  For every $\omega \in \D_0([0,1],\R)$ and $0 \leq u \leq 1$, we define the shifted function $\omega^{(u)}$ by
\[\omega^{(u)}(t)=\begin{cases} \omega(u+t) -\omega(u)
 \qquad \qquad  \qquad \qquad \quad \, \, \, \, \, \textrm{if }
u+t \leq 1,
\\
\omega(u+t-1)+\omega(1)- \omega(u) \qquad \quad
\qquad \textrm{  if } u+t \geq 1.
\end{cases}
\]
We shall also need the notation $g_1(\omega)=\inf\{ t \in [0,1]; \omega(t-)
\wedge \omega(t)= \inf_{[0,1]} \omega\}$. The shifted function $\omega^{(g_{1}(\omega))}$ is usually called the Vervaat transform of $\omega$.

\begin{lem}
\label{lem:shift}
Let $B^{br}$ be a standard Brownian bridge and $B$ an independent standard Brownian motion. Set $\tau=g_{1}(B^{br})$. Then
\begin{align*}
\left( B^{br,(\tau)}, B^{(\tau)} \right)   \quad \mathop{=}^{(d)} \quad \left(\mathbbm{e}, B' \right),
\end{align*}
where $\mathbbm{e}$ is a normalized Brownian excursion and  $B'$ is a standard Brownian motion independent of $\mathbbm{e}$.
\end{lem} 

\begin{proof}
Since $B$ and $B^{br}$ are independent, it readily follows that $B^{(\tau)}$  has the law of a standard Brownian motion, and is independent of $(\tau,B^{br})$, and therefore is independent of $B^{br, (\tau)}$. On the other hand, $B^{br, (\tau)}$ has the law of $\mathbbm{e}$ (see e.g.~\cite{Ver79}). The result follows.
\end{proof}

\begin{proof}[Proof of Theorem \ref{thm:processus} (i)]
We keep the notation of Lemma \ref{lem:shift}, and we also let $(S^{br,n},J^{n})=(S^{br,n}_{nt},J^{n}_{nt})_{0 \leq t \leq 1}$ be a random variable distributed as $(S_{nt},J_{nt}^{\cA} - nt\mu_{\cA})_{0 \leq t \leq 1}$ conditionally given $S_{n}=-1$. We set $\tau_{n}=g_{1}(S^{br,n})$. It is well-known (see e.g.~\cite{Ver79}) that $S^{br,n,(\tau_{n})}$ has the same distribution as $(W_{nt}(\cT_{n}))_{0 \leq t \leq 1}$. It follows that
\[ \left( S^{br,n,(\tau_{n})}, J^{n,(\tau_{n})}\right)  \quad \mathop{=}^{(d)} \quad \left( W_{nt}(\cT_{n}), K^{\cA}_{nt}(\cT_n)-nt \mu_\cA \right)_{0 \leq  t \leq 1}.\]
Since $B^{br}$ and $B$ are almost surely continuous at $\tau$, by Lemma \ref{lem:pont} and standard continuity properties of the Vervaat transform, it follows that
\[ \left( \frac{W_{nt}(\cT_{n})}{\sqrt{n}} , \frac{K^{\cA}_{nt}(\cT_n)-nt \mu_\cA}{\sqrt{n}} \right) _{0 \leq t \leq 1}   \quad \mathop{\longrightarrow}^{(d)}_{n \rightarrow \infty} \quad \left( \sigma B^{br,(\tau)}_t, \frac{\mu_{\cA} m_{\cA}}{\sigma} B^{br,(\tau)}_t + \gamma_{\cA} B'^{(\tau)}_t \right)_{0 \leq t \leq 1}.\]
By Lemma \ref{lem:shift}, this last process has the same distribution as $( \sigma \mathbbm{e}_t, \frac{\sum_{i \in \cA} (i-1) \mu_i}{\sigma} \mathbbm{e}_t + \gamma_{\cA} B'_t)$, and this completes the proof.
\end{proof}

\subsection{Contour exploration}
\label{proc2}

We are now interested in the evolution of the number of $\cA$-vertices in conditioned GW trees for the contour visit counting process, and establish in particular Theorem \ref{thm:processus} (ii).
The idea of the proof is to obtain a relation between the counting process $N^\cA$ for the contour process and the counting process $K^{\cA}$ for the depth-first search order.

In this direction, if $T$ is a tree with $n$ vertices,  for every $0 \leq k \leq 2n-2$, we denote by $b_{k}(T)$ the number of \textit{different} vertices visited by the contour process  $C(T)$ up to time $k$. We set $b_{k}(T)=b_{2n-2}(T)$ for $k \geq 2n-2$, and $b_{t}(T)=b_{\lfloor t \rfloor}(T)$ for $t \geq 0$. It turns out that the following simple deterministic relation holds between $b(T)$ and $C(T)$.

\begin{lem}
\label{lem:relationbC}
Let $T$ be a tree with $n$ vertices. Then, for every $0 \leq k \leq 2n-2$,
\begin{align*}
b_{k}(T) =  1 + \frac{k+C_{k}(T)}{2}.
\end{align*}
\end{lem}

\begin{proof}
We show that the result holds for $k=0$, and that if it holds at time $0 \leq k \leq 2n-3$ then it holds at time $k+1$. For $0 \leq k \leq 2n-2$, let $u_k$ be the vertex visited by the contour process at time $k$. First, at time $k=0$, the root is the only vertex visited and $b_{0}(T) = 1$. Now assume that the result holds until time $0 \leq k \leq 2n-3$. Then we see that $u_{k+1}$ is visited for the first time at time $k+1$ if and only if the contour process goes up between $u_k$ and $u_{k+1}$.
Therefore, $b_{k+1}(T) = b_{k}(T) + 1$ if $C_{k+1}(T) = C_{k}(T)+1$ and $b_{k+1}(T) = b_{k}(T)$ if $C_{k+1}(T) = C_{k}(T)-1$. In both cases, the formula is also valid at time $k+1$.
\end{proof}

We are now in position to establish Theorem \ref{thm:processus} (ii).

\begin{proof}[Proof of Theorem \ref{thm:processus} (ii).]
First, by Theorem \ref{thm:processus} (i) and Lemma \ref{lem:relationbC}, the convergence

\begin{equation}
\label{eq:cvjointe}
\left(\frac{C_{2nt}(\mathcal{T}_n)}{\sqrt{n}}, \frac{b_{2nt}(\mathcal{T}_n) - nt}{\sqrt{n}}, \frac{K^{\cA}_{nt}(\cT_n) - nt \mu_\cA}{\sqrt{n}} \right)_{0 \leq t \leq 1} \rightarrow \left(\frac{2}{\sigma} \mathbbm{e}_t, \frac{1}{\sigma} \mathbbm{e}_t, \frac{\mu_{\cA} m_{\cA}}{\sigma}   \mathbbm{e}_t + \gamma_{\cA} B_t \right)_{0 \leq t \leq 1}
\end{equation}
holds jointly in distribution in $\D([0,1],\R^{3})$, where $\mathbbm{e}$ is a normalized Brownian excursion and $B$ is an independent standard Brownian motion. In particular, the convergence
\begin{equation}
\label{eq:cvprob}
\left( \frac{b_{2nt}(\mathcal{T}_n)}{n} \right)_{0 \leq t \leq 1}   \quad \mathop{\longrightarrow}_{n \rightarrow \infty} \quad  (t)_{0 \leq t \leq 1}
\end{equation}
holds in probability.

Next, for every $t \in [0,1]$, observe that $N^{\cA}_{2nt}(\mathcal{T}_n) = K^{\cA}_{b_{2nt}(\mathcal{T}_n)}(\cT_n)$, so that
\[
\frac{N^{\cA}_{2nt}(\mathcal{T}_n)-nt \mu_\cA}{\sqrt{n}} = \frac{K^{\cA}_{b_{2nt}(\mathcal{T}_n)}(\cT_n)-nt \mu_\cA}{\sqrt{n}}=\frac{K^{\cA}_{b_{2nt}(\mathcal{T}_n)}(\cT_n)-b_{2nt}(\mathcal{T}_n) \mu_\cA}{\sqrt{n}} + \mu_\cA \frac{b_{2nt}(\mathcal{T}_n)-nt}{\sqrt{n}}.\]
By \eqref{eq:cvjointe}  and \eqref{eq:cvprob}, it follows that the convergence
\begin{align*}
\left(\frac{N^\cA_{2nt}(\mathcal{T}_n)-nt \mu_\cA}{\sqrt{n}} \right)_{0 \leq t \leq 1} \rightarrow \left(\frac{\mu_\cA m_\cA}{\sigma}\mathbbm{e}_t + \gamma_\cA B_t + \frac{\mu_\cA}{\sigma} \mathbbm{e}_t \right)_{0 \leq t \leq 1}
\end{align*}
holds in distribution, jointly with \eqref{eq:cvjointe}.  Since $\mu_\cA m_\cA=\sum_{i \in \cA} (i-1) \mu_i$, this completes the proof.
\end{proof}

\subsection{Extension to multiple passages}
\label{sec:secgen}

In Theorem \ref{thm:processus} (ii), the process $N^\cA$ counts $\cA$-vertices the first time they are visited by the contour exploration. In this Section, we are interested in what happens when instead we count vertices at later visit times.  In this direction, if $T$ is a tree, for every and $1 \leq k \leq i+1$ and $0 \leq \ell \leq 2 |T|$, we denote by $N^{i,k}_{\ell}(T)$ the number of vertices of outdegree $i$ visited at least $k$ times by the contour exploration of $T$ between times $0$ and $\ell$. Finally, for $i \geq0$, we set $N^{i}=N^{\{i\}}$ to simplify notation.

As before, we fix a critical distribution $\mu$ with finite positive variance $\sigma^2$, and we let $\mathcal{T}_n$ denote a $\mu$-GW tree conditioned on having $n$ vertices.

\begin{thm}
\label{thm:cvNik}
We have
\begin{align*}
&\left( \frac{C_{2nt}(\mathcal{T}_n)}{\sqrt{n}}, \frac{N^{i}_{2nt}(\mathcal{T}_n)- nt \mu_i}{\sqrt{n}}, \frac{N^{i,k}_{2nt}(\mathcal{T}_n)- nt \mu_i}{\sqrt{n}} \right)_{0 \leq t \leq 1} \\
&  \qquad \qquad  \qquad  \mathop{\longrightarrow}^{(d)}_{n \rightarrow \infty} \quad \left(\frac{2}{\sigma} \mathbbm{e}_t, \frac{i \mu_i}{\sigma} \mathbbm{e}_t + \gamma_i B_t, \frac{(i-2(k-1))\mu_i}{\sigma} \mathbbm{e}_t + \gamma_i B_t \right)_{0 \leq t \leq 1}
\end{align*}
where $B$ is a standard Brownian motion independent of $\mathbbm{e}$ and $\gamma_{i}=\sqrt{\mu_{i} (1-\mu_{i}) - \frac{1}{\sigma^2} ((i-1)\mu_{i}  )^2}$.
\end{thm}

The main ingredient of the proof is a relation between $N^i(T)$ and $N^{i,j}(T)$, for which we need to introduce some notation. If $T$ is a tree, for $u \in T$ and $1 \leq j \leq i$, we denote by $A^{i,j}_u(T)$ the number of ancestors of $u$  in $T$ with $i$ children whose $j$th child is an ancestor of $u$.  For $0 \leq t \leq 2|T|-2$, denote by $u_t(T)$ the vertex visited at time $\lfloor t \rfloor$ by contour exploration. Then, for every $0 \leq \ell \leq 2|T|-2$, observe that
\begin{equation}
\label{eq:lienN}
N^i_{\ell}(T)-N^{i,k}_{\ell}(T) = \sum_{1 \leq j \leq k-1} A^{i,j}_{u_{\ell}(T)}(T)
\end{equation}
because  $i$-vertices of $T$ that have been visited at least once up to time $\ell$, but not $k$ times yet, are necessarily ancestors of $u_{\ell}(T)$. Indeed, all the subtrees attached to a strict ancestor of $u_{\ell}(T)$ have either been completely visited or not visited at all (except the subtrees containing $u_{\ell}(T)$).

The following result, which is of independent interest, will allow to control the asymptotic behaviour of  $ A_u^{i,j}(\mathcal{T}_n)$, when the height of $u$ is large enough. See \cite{MM03b} for other bounds on $ A^{i,j}(\mathcal{T}_n)$ under an additional finite exponential moment assumption.
For a nonnegative sequence $(r_{n})$, we write $r_{n}=oe(n)$  if there exist $C, \varepsilon>0$ such that $r_n \leq C e^{-n^\varepsilon}$ for every $n \geq 1$.

\begin{prop}
\label{prop:ancestors}
Fix $i \geq 1$. Then
\[
\mathbb{P} \left( \exists u \in \mathcal{T}_n, \exists j \in \llbracket 1, i \rrbracket : |u| \geq n^{1/10}, \left|\frac{A^{i,j}_u(\mathcal{T}_n)}{|u|} - \mu_i \right| \geq \frac{\mu_i}{|u|^{1/100}} \right)  =oe(n).\]
\end{prop}

Before proving this bound, let us explain how Theorem \ref{thm:cvNik} follows.

\begin{proof}[Proof of Theorem \ref{thm:cvNik} from Proposition \ref{prop:ancestors}] We will repeatedly use the identity $C_{\ell}(\mathcal{T}_n)=|u_{\ell}(\mathcal{T}_n)|$ for every $0 \leq \ell \leq 2n-2$. We first check that
\begin{equation}
\label{eq:controleAij} \max_{1 \leq j \leq i} \sup_{0 \leq \ell \leq 2n-2} \left|  \mu_{i} \frac{C_{\ell}(\mathcal{T}_n)}{\sqrt{n}}- \frac{A^{i,j}_{u_{\ell}(\cT_n)}(\mathcal{T}_n)}{\sqrt{n}}  \right|   \quad \mathop{\longrightarrow}^{(\P)}_{n \rightarrow \infty} \quad  0
\end{equation}
First, since for every $\ell$, $ A^{i,j}_{u_{\ell}(\cT_n)}(\mathcal{T}_n) \leq |u_\ell(\mathcal{T}_n)|$, we may restrict our study without loss of generality to the times $\ell$ such that $|u_{\ell}(\cT_n)| \geq n^{1/10}$. Indeed, uniformly for $1 \leq j \leq i$, for $\ell$ such that $|u_{\ell}(\cT_n)| < n^{1/10}$, we have $n^{-1/2}|\mu_i C_\ell(\cT_n) - A^{i,j}_{u_\ell(\cT_n)}(\cT_n)| \leq 2 n^{1/10-1/2} = 2 n^{-2/5}$.

By Proposition \ref{prop:ancestors}, for every $n$ sufficiently large and $\ell$ such that   $|u_{\ell}(\cT_n)| \geq n^{1/10}$, we have, with probability tending to $1$ as $n \rightarrow  \infty$, uniformly in $j$, $ \left| A^{i,j}_{u_{\ell}(\cT_n)}(\cT_n) - \mu_i |u_{\ell}(\cT_n) | \right| < \mu_i | u_{\ell}(\cT_n)|^{99/100}$. By Theorem \ref{thm:CW}, $\max_{0 \leq \ell \leq 2n-2} C_{\ell}(\cT_n)/\sqrt{n}$ converges in probability as $n \to \infty$, so that we have $\max_{0 \leq \ell \leq 2n-2} |u_{\ell}(\cT_n)|^{99/100}/\sqrt{n} \to 0$. This entails \eqref{eq:controleAij}. 

Now, using \eqref{eq:lienN}, for $0 \leq t \leq 1$, write  
 \[\frac{N^{i,k}_{2nt}(\mathcal{T}_n)- nt \mu_i}{\sqrt{n}}=\frac{N^{i}_{2nt}(\mathcal{T}_n)-nt\mu_i}{\sqrt{n}} - \sum_{1 \leq j \leq k-1} \frac{A^{i,j}_{u_{2nt}(\mathcal{T}_n)}(\cT_n)}{\sqrt{n}}\]
Hence, by combining Theorem \ref{thm:processus} (ii) with \eqref{eq:controleAij}, we get

\begin{align*}
&\left( \frac{C_{2nt}(\mathcal{T}_n)}{\sqrt{n}}, \frac{N^{i}_{2nt}(\mathcal{T}_n)- nt \mu_i}{\sqrt{n}}, \frac{N^{i,k}_{2nt}(\mathcal{T}_n)- nt \mu_i}{\sqrt{n}} \right)_{0 \leq t \leq 1} \\
&  \qquad \qquad  \qquad  \mathop{\longrightarrow}^{(d)}_{n \rightarrow \infty} \quad \left(\frac{2}{\sigma} \mathbbm{e}_t, \frac{i \mu_i}{\sigma} \mathbbm{e}_t + \gamma_i B_t, \frac{i \mu_i}{\sigma} \mathbbm{e}_t + \gamma_i B_t -  (k-1) \mu_{i} \frac{2}{\sigma} \mathbbm{e}_{t}\right)_{0 \leq t \leq 1} 
\end{align*}
where $B$ is a standard Brownian motion and $\gamma_{i}= \sqrt{\mu_{i} (1-\mu_{i}) - \frac{1}{\sigma^2} ((i-1)\mu_{i}  )^2}$, which gives the desired result.
\end{proof}

We now get into the proof of Proposition \ref{prop:ancestors}.

\begin{proof}[Proof of Proposition \ref{prop:ancestors}]
First, observe that if $ \mathcal{T}$ is a nonconditioned $\mu$-GW tree, then

\begin{align*}
&\mathbb{P} \left( \exists u \in \mathcal{T}_n, \exists j \in \llbracket 1, i \rrbracket : |u| \geq n^{1/10}, \left|\frac{A^{i,j}_u(\cT_n)}{|u|} - \mu_i \right| \geq \frac{\mu_i}{|u|^{1/100}} \right)  \\
&  \qquad  \qquad \qquad  \qquad \leq \frac{1}{\P(| \mathcal{T}|=n )} \sum_{k= \lceil n^{1/10} \rceil}^n \sum_{j=1}^{i} \E \left[  \sum_{|u|=k} \mathbbm{1}_{\left|\frac{A^{i,j}_u(\cT)}{|u|} - \mu_i \right| \geq \frac{\mu_i}{|u|^{1/100}}}\right].
\end{align*}

In order to compute these expectations, let us mention the existence of the local limit $\cT^*$ of the trees $\cT_n$. This limit is defined as the random variable on the set of infinite trees satisfying, for any $r \geq 0$,
\begin{align*}
B_r\left(\cT_n \right) \underset{n \rightarrow \infty}{\rightarrow} B_r \left( \cT^* \right),
\end{align*}
where $B_r$ denotes the ball of radius $r$ centered at the root for the graph distance (all edges of the tree having length $1$).
$\cT^*$ is an infinite tree called Kesten's tree, made of a unique infinite branch on which i.i.d. $\mu$-Galton-Watson trees are planted (see \cite{Kes86} for details). The local behaviour of the trees $\cT_n$ can be deduced from the properties of this infinite tree; in particular, a standard size-biasing identity \emph{à la} Lyons--Pemantle--Peres \cite{LPP95} (see \cite[Eq.~(23)]{Duq09} for a precise statement) gives

\begin{align*}
\E \left[  \sum_{|u|=k} \mathbbm{1}_{\left|\frac{A^{i,j}_u(\cT)}{|u|} - \mu_i \right| \geq \frac{\mu_i}{|u|^{1/100}}}\right] = \E \left[ \mathbbm{1}_{\left|\frac{A^{i,j}_{U_k(\cT^*)}(\cT^*)}{k} - \mu_i \right| \geq \frac{\mu_i}{k^{1/100}}}\right]
= \P \left( \left| \frac{Bin(k,\mu_{i})}{k}-\mu_{i} \right| \geq  \frac{\mu_i}{k^{1/100}}\right),
\end{align*}
where $U_k(\cT^*)$ denotes the vertex of the unique infinite branch of $\cT^*$ at height $k$. In particular, this expectation does not depend on $j$.

By \eqref{eq:NB} (applied with $\cB=\Z_{+}$),  we therefore have, for some constant $C$:
\begin{align*}
&\mathbb{P} \left( \exists u \in \mathcal{T}_n, \exists j \in \llbracket 1, i \rrbracket : |u| \geq n^{1/10}, \left|\frac{A^{i,j}_u(\cT_n)}{|u|} - \mu_i \right| \geq \frac{\mu_i}{|u|^{1/100}} \right) \\
& \qquad \qquad\qquad\qquad\qquad \leq C i n^{3/2 }  \sum_{k= \lceil n^{1/10} \rceil}^n  \P \left( \left| \frac{Bin(k,\mu_{i})}{k}-\mu_{i} \right| \geq  \frac{\mu_i}{k^{1/100}}\right) \\
& \qquad \qquad\qquad\qquad\qquad  \leq C i n^{3/2 }  \sum_{k= \lceil n^{1/10} \rceil}^n  2 \exp \left( -2 k \left( \mu_i k^{-1/100} \right)^2 \right)
\end{align*}
where the last line is obtained by using Hoeffding inequality. Thus,
\begin{align*}
&\mathbb{P} \left( \exists u \in \mathcal{T}_n, \exists j \in \llbracket 1, i \rrbracket : |u| \geq n^{1/10}, \left|\frac{A^{i,j}_u(\cT_n)}{|u|} - \mu_i \right| \geq \frac{\mu_i}{|u|^{1/100}} \right) \\
& \qquad \qquad\qquad\qquad\qquad  \leq 2 C i n^{3/2 }  \sum_{k= \lceil n^{1/10} \rceil}^n   \exp \left( -2 \mu_i^2 k^{49/50} \right) = oe(n).
\end{align*}
The desired result follows.
\end{proof}

Finally, let us remark that the estimate of Proposition \ref{prop:ancestors} is strong enough to get the following refinement of Theorem \ref{thm:cvNik} (whose proof is left to the reader):

\begin{thm}
\label{thm:cvNik2}
Let $k : \Z_+ \rightarrow \Z_+$ such that, for $i \in \Z_+$, $1 \leq k(i) \leq i+1$. Let $\cA \subset \Z_+$. Then the following convergence holds in distribution :
\begin{align*}
&\left( \frac{C_{2nt}(\mathcal{T}_n)}{\sqrt{n}}, \frac{N^{\cA}_{2nt}(\mathcal{T}_n)- nt \mu_\cA}{\sqrt{n}}, \frac{\sum_{i \in \cA} N^{i, k(i)}_{2nt}(\mathcal{T}_n)- nt \mu_i}{\sqrt{n}} \right)_{0 \leq t \leq 1} \\
&  \qquad \qquad  \qquad  \mathop{\longrightarrow}^{(d)}_{n \rightarrow \infty} \quad \left(\frac{2}{\sigma} \mathbbm{e}_t, \frac{\sum_{i \in \cA} i \mu_i}{\sigma} \mathbbm{e}_t + \gamma_\cA B_t, \frac{\sum_{i \in \cA}\left(i -2(k(i)-1)\right) \mu_i}{\sigma} \mathbbm{e}_t + \gamma_\cA B_t \right)_{0 \leq t \leq 1} 
\end{align*}
where $B$ is a standard Brownian motion and $\gamma_{\cA}= \sqrt{\mu_{\cA} (1-\mu_{\cA}) - \frac{1}{\sigma^2} (\sum_{i \in \cA} (i-1)\mu_{i})^2}$.
\end{thm}

\section{Asymptotic normality of outdegrees in large Galton-Watson trees}
\label{sec:asymptotic_normality}

The main goal of this Section is to prove Theorem \ref{thm:normality} (i) and (ii). We fix a critical offspring distribution $\mu$ with finite positive variance $\sigma^2$, and $\cA, \cB \subset \mathbb{Z}_{+}$ such that $\mu_\cB>0$. If $T$ is a tree, recall that $N^{\cA}(T)$ is the number of $\cA$-vertices in $T$, and that  $\mathcal{T}_n^{\cB}$ is a $\mu$-GW tree conditioned to have $n$ $\cB$-vertices. 
In the sequel, $\mathcal{T}$ is a nonconditioned $\mu$-GW tree. We also assume for technical convenience that $p_\cB$ and $p_{\cB^c}$ are both aperiodic (but the results carry through in the general setting with mild modifications).

\subsection{Expectation of $N^{\cA}(\mathcal{T}_n^{\cB})$}
\label{subsection:expectation}
Our goal is here to prove Theorem \ref{thm:normality} (i). For every $n \geq 1$, define the interval $I_n := \llbracket \frac{n}{\mu_\cB} - n^{3/4}, \frac{n}{\mu_\cB} +n^{3/4} \rrbracket$. The proof relies on the following estimates.

\begin{lem}
\label{lem:estim}
We have:
\begin{enumerate}
\item[(i)] $\E \left(  N^{\Z_+} \left( \mathcal{T}_n^\cB \right) \mathds{1}_{N^{\Z_+} \left( \mathcal{T}_n^\cB \right) \notin  I_n} \right) = oe(n)$;
\item[(ii)] $\P \left( \left. \left| \frac{N^\cA \left( \mathcal{T}_n^\cB \right)}{n} - \frac{\mu_\cA}{\mu_\cB} \right| \geq n^{-1/5} \right| N^{\Z_+}(\mathcal{T}_n^\cB)\in I_n \right) \rightarrow 0$ as $n \rightarrow \infty$.
\end{enumerate}
\end{lem}

\begin{proof}[Proof of Theorem \ref{thm:normality} (i) using Lemma \ref{lem:estim}]
Start by writing the quantity $\E [ N^\cA ( \mathcal{T}_n^\cB ) ]$ as
\begin{align}
\label{align:fin}
\E [ N^\cA ( \mathcal{T}_n^\cB ) ]=\P ( N^{\Z_+} ( \mathcal{T}_n^\cB ) \in  I_n ) \E [  N^\cA ( \mathcal{T}_n^\cB ) | N^{\Z_+} ( \mathcal{T}_n^\cB ) \in  I_n  ] + 
 \E [ N^\cA ( \mathcal{T}_n^\cB ) \mathds{1}_{N^{\Z_+} ( \mathcal{T}_n^\cB ) \notin  I_n} ].
\end{align}
Observe that $\E [ N^\cA ( \mathcal{T}_n^\cB ) \mathds{1}_{N^{\Z_+} ( \mathcal{T}_n^\cB ) \notin  I_n} ] \leq \E [ N^{\Z_+} ( \mathcal{T}_n^\cB ) \mathds{1}_{N^{\Z_+} ( \mathcal{T}_n^\cB ) \notin  I_n} ]=oe(n)$ by Lemma \ref{lem:estim} (i). In order to bound the first term in the sum of \eqref{align:fin}, bound $\Big|\frac{1}{n} \E \left[  N^\cA \left( \mathcal{T}_n^\cB \right) | N^{\Z_+} \left( \mathcal{T}_n^\cB \right) \in  I_n  \right] - \frac{\mu_\cA}{\mu_\cB} \Big|$ from above by 
\begin{align*}
 \frac{1}{n^{1/5}} + \left( \frac{\sup I_n}{n} + \frac{\mu_\cA}{\mu_\cB} \right) \P \left( \left. \left| \frac{N^\cA ( \mathcal{T}_n^\cB  )}{n} - \frac{\mu_\cA}{\mu_\cB} \right| \geq  \frac{1}{n^{1/5}} \right| N^{\Z_+}(\mathcal{T}_n^\cB )\in I_n \right).
\end{align*}
This last quantity tends to $0$ as $n \rightarrow \infty$  by Lemma \ref{lem:estim} (ii) and since ${\sup I_n}/{n} \rightarrow {1}/{\mu_\cB}$.
In order to complete the proof, it remains to observe that since  $N^{\Z_+}(\mathcal{T}_n^\cB) \geq n$, Lemma \ref{lem:estim} (i) implies that $\P \left( N^{\Z_+} \left( \mathcal{T}_n^\cB \right) \notin I_n \right) \rightarrow 0$.
\end{proof}

\begin{proof}[Proof of Lemma \ref{lem:estim}]
First, remark that
\begin{align*}
\E \left[  N^{\Z_+} \left( \mathcal{T}_n^\cB \right) \mathds{1}_{N^{\Z_+} \left( \mathcal{T}_n^\cB \right) \notin  I_n} \right] &= \sum_{\substack{k \notin I_n \\ k \geq n}} k \, \P \left( N^{\Z_+} \left( \mathcal{T}_n^\cB \right) = k \right) \nonumber \\ 
&= \frac{1}{\P \left( N^\cB(\cT)=n \right)} \sum_{\substack{k \notin I_n \\ k \geq n}} k \, \P \left( N^{\Z_+} \left( \cT \right) = k, N^\cB \left( \cT \right) = n \right) \nonumber \\
& \leq \frac{1}{\P \left( N^\cB(\cT)=n \right)} \sum_{\substack{k \notin I_n \\ k \geq n}} \P \left( J_k^\cB = n \right) \text{ by Proposition \ref{prop:joint}}.
\end{align*}
We now use the fact that, for any $k$, $J_k^\cB$ has a binomial distribution of parameters $\left( k,\mu_\cB \right)$. Remark that, if $k \notin I_n$, then $|n-k\mu_\cB| \geq k^{3/5}$. Hence, by Hoeffding's inequality, for $k \not \in I_{n}$, $\P( J_k^\cB = n) \leq \P( |J_k^\cB - k \mu_{\cB}| \geq k^{3/5}) \leq 2 e^{-2 k^{1/5}}$. Therefore $\sum_{{k \notin I_n, k \geq n}} \P ( J_k^\cB = n )=oe(n)$. (i) follows by \eqref{eq:NB} (applied with $\cB=\Z_{+}$),

For (ii), we use the fact that
\begin{align*}
&\P \left( \left. \left| \frac{N^\cA \left( \mathcal{T}_n^\cB \right)}{n} - \frac{\mu_\cA}{\mu_\cB} \right| \geq n^{-1/4} \right| N^{\Z_+}(\mathcal{T}_n^\cB)\in I_n \right) \\
&=\frac{1}{\P \left( \left. N^\cB(\cT)=n \right| N^{\Z_+}(\cT) \in I_n \right)}\P \left( \left. \left| \frac{N^\cA \left( \cT \right)}{n} - \frac{\mu_\cA}{\mu_\cB} \right| \geq n^{-1/5}, N^\cB(\cT)=n \right| N^{\Z_+}(\cT)\in I_n \right)
\end{align*}
Note that
\begin{align*}
\frac{1}{\P \left( \left. N^\cB(\cT)=n \right| N^{\Z_+}(\cT) \in I_n \right)} = \frac{\P \left( N^{\Z_+}(\cT) \in I_n \right)}{\P \left( N^\cB(\cT)=n, N^{\Z_+}(\cT) \in I_n \right)} \leq \frac{1}{\P \left( N^\cB(\cT)=n, N^{\Z_+}(\cT) = \lfloor \frac{n}{\mu_\cB} \rfloor \right)}
\end{align*}
which grows at most polynomially in $n$ according to Lemma \ref{lem:asymp} (ii). The second assertion now follows from the fact that
\begin{align*}
\P \left( \left. \left| \frac{N^\cA \left( \cT \right)}{n} - \frac{\mu_\cA}{\mu_\cB} \right| \geq n^{-1/5} \right| N^{\Z_+}(\cT)\in I_n \right) &\leq \underset{k \in I_n}{\sup} \P \left( \left. \left| \frac{N^\cA \left( \cT \right)}{n} - \frac{\mu_\cA}{\mu_\cB} \right| \geq n^{-1/5} \right| N^{\Z_+}(\cT) = k \right).
\end{align*}

In virtue of \eqref{eq:NB} (applied with $\cB=\Z_{+}$), it suffices to check that $ \P ( | \frac{N^\cA \left( \cT \right)}{n} - \frac{\mu_\cA}{\mu_\cB} | \geq n^{-1/5}, N^{\Z_+}(\cT) = k )=oe(n)$ when $k \in I_{n}$. By  Proposition \ref{prop:joint},
\[ \P \left( \left| \frac{N^\cA \left( \cT \right)}{n} - \frac{\mu_\cA}{\mu_\cB} \right| \geq n^{-1/5}, N^{\Z_+}(\cT) = k \right) \leq \P \left( \left| J_{k}^\cA - n \frac{\mu_\cA}{\mu_\cB} \right| \geq n^{4/5}\right).\]
When $k \in I_{n}$, this last quantity is bounded from above by  $\P ( | J_{k}^\cA - k\mu_\cA | \geq n^{4/5}-\mu_\cA n^{3/4} )$, which is $oe(n)$ since $J_k^\cA$ has a binomial distribution of parameters $\left( k,\mu_\cA \right)$. This proves (ii).
\end{proof}

\subsection{Asymptotic normality of $N^{\cA}(\mathcal{T}_k^{\cB})$}
\label{ssec:normality}

The first step is to establish the following local version of  Theorem \ref{thm:normality} when  $\cA=\Z_{+}$.

\begin{prop}
\label{loool}
As $k \rightarrow \infty$,
\begin{align*}
\P \left(  N^{\Z_+} \left(\mathcal{T}^{\cB}_k \right)= \lfloor {k}/{\mu_{\cB}} + \sqrt{k} y \rfloor \right) \sim \sqrt{\frac{\mu_\cB^3}{2 \pi \gamma_\cB^2}} \, \, \, \frac{1}{\sqrt{k}} \exp \left( - \frac{\mu_\cB^3}{\gamma_\cB^2} \frac{y^2}{2} \right),
\end{align*}
uniformly for $y$ in a compact subset of $\mathbb{R}$.
\end{prop}

It is standard that this implies the following asymptotic normality:
\begin{equation}
\label{eq:NZ}
\frac{N^{\Z_+}(\mathcal{T}^{\cB}_k)-k/\mu_{\cB}}{\sqrt{k}} \overset{d}{\rightarrow} \mathcal{N} (0, \frac{\gamma^2_{\cB}}{\mu^3_{\cB}}).
\end{equation}

\begin{proof}[Proof of Proposition \ref{loool}]
By Lemma \ref{lem:asymp} (ii), we have as $n \rightarrow \infty$, uniformly for $c$ in a compact subset of $\R$,
\begin{align}
\mathbb{P} \left( N^{\Z_+}(\cT)=n, N^{\cB}(\cT)= k_{n}(c) \right) & \sim \frac{1}{{2 \pi \sigma  \gamma_\cB}} \frac{1}{n^2}  \exp \left( - \frac{1}{\gamma_{\cB}^{2} } \frac{c^{2}}{2}\right).\label{eq:equivalent}
\end{align}
By using  \eqref{eq:NB},
we have
\[
\mathbb{P}(N^{\Z_+}(\cT)=n | N^{\cB}(\cT)=k_n(c)) \sim \frac{\mu_{\cB}}{\gamma_{\cB}\sqrt{2 \pi n}} \exp \left( - \frac{1}{\gamma_{\cB}^{2} } \frac{c^{2}}{2}\right).
\]
Then observe that for $y \in \R$, as $n,k \rightarrow \infty$, it is equivalent to write
$n={k}/{\mu_{\cB}}+ y \sqrt{k} + O(1)$ and  $k=n \mu_{\cB} - y \sqrt{n} \mu_{\cB}^{3/2} + O(1)$.
Hence
\begin{equation*}
\mathbb{P} \left( N^{\Z_+}(\cT)= \lfloor \frac{k}{\mu_\cB} + y \sqrt{k} \rfloor | N^{\cB}(\cT)=k \right) \sim \frac{\mu_{\cB}^{3/2}}{\gamma_{\cB}\sqrt{2 \pi k}} \exp \left( - \frac{\mu_{\cB}^{3}}{\gamma_{\cB}^{2} } \frac{y^{2}}{2}\right).
\end{equation*}
This completes the proof.
\end{proof}

We are now in position to establish Theorem \ref{thm:normality} (ii), which will be a consequence of the following estimate. 

\begin{lem}
\label{lem:normality}
Let $\cA, \cB \subset \Z_+$ such that the quantities $\mu_{\cA \cap \cB}$, $\mu_{\cA \backslash \cB}$, $\mu_{\cB \backslash \cA}$, $\mu_{\cA^c \cap \cB^c} $ are all positive.Then there exists $\sigma_{\cA,\cB}^2 > 0$, $C_{\cA,\cB} \in \R$ such that for fixed $u,v \in \R \cup \{+\infty, -\infty \}, u<v$ and $y \in \R$, we have, as $k \rightarrow \infty$,
\begin{align*}
\mathbb{P} \left( \left. \frac{N^{\cA}(\mathcal{T}^{\cB}_k) - k \frac{\mu_{\cA}}{\mu_{\cB}}}{\sqrt{k}} \in (u,v) \right| N^{\Z_+}(\mathcal{T}^{\cB}_k) = \lfloor k/\mu_{\cB} + \sqrt{k} y \rfloor \right) \sim  \frac{1}{\sqrt{2 \pi \sigma_{\cA,\cB}^2}} \displaystyle\int_{u}^{v} e^{-\frac{1}{2 \sigma_{\cA,\cB}^2}{\left(z- C_{\cA,\cB} y \right)^2}} dz.
\end{align*}
\end{lem}

\begin{proof}[Proof of Theorem \ref{thm:normality} (ii), using Lemma \ref{lem:normality}] First assume that the quantities $\mu_{\cA \cap \cB}$, $\mu_{\cA \backslash \cB}$, $\mu_{\cB \backslash \cA}$, $\mu_{\cA^c \cap \cB^c} $ are all positive.
Fix $u<v$. For $y \in \mathbb{R}$ and $k \in \Z_+$, set \[f_k(y) =  \mathbb{P} \left(  \frac{N^{\cA}(\mathcal{T}^{\cB}_k) - k \frac{\mu_{\cA}}{\mu_{\cB}}}{\sqrt{k}} \in (u,v) , N^{\Z_+}(\mathcal{T}^{\cB}_k) = \lfloor k/\mu_{\cB} + \sqrt{k} y \rfloor \right) \sqrt{k}\]
and remark that
$
\mathbb{P} ( ({N^{\cA}(\mathcal{T}^{\cB}_k) - k \frac{\mu_{\cA}}{\mu_{\cB}}})/{\sqrt{k}} \in (u,v) ) = \int_{\mathbb{R}}  f_{k}(y) dy$. Also, for $y,z \in \mathbb{R}$ define $g(y,z)$ by
\[g(y,z) = \frac{1}{\sqrt{2 \pi \gamma^2}} e^{-\frac{y^2}{2 \gamma^2}} \frac{1}{\sqrt{2 \pi \sigma_{\cA,\cB}^2}}  e^{-\frac{1}{{2 \sigma_{\cA,\cB}^2}}{\left(z-C_{\cA,\cB}y \right)^2}}.\]
where $\gamma^2 = {\gamma_\cB^2}/{\mu_\cB^3}$.
Observe that $\int_{\mathbb{R}^2} g(y,z)dy dz=1$.
Then, by Proposition \ref{loool} and Lemma \ref{lem:normality}, $f_k(y)$ converges pointwise, as $k \rightarrow \infty$, to
$\int_u^v g(y,z) dz$. Hence, by Fatou's lemma and Fubini-Tonnelli's theorem,
\[\underset{k \rightarrow \infty}{\liminf} \, \mathbb{P} \left( \frac{N^{\mathcal{A}}(\cT^\mathcal{B}_k) - k \frac{\mu_{\mathcal{A}}}{\mu_{\mathcal{B}}}}{\sqrt{k}} \in (u,v) \right) \geq \int_u^v \left[ \int_{\R} g(y,z) dy \right] dz.
\]
By Portmanteau theorem, if $(X_k)$ is a sequence of real-valued random variables such that for every $u<v$, $\liminf_{k \to \infty} \mathbb{P}(u < X_k < v) \geq \mathbb{P}(u < X < v)$ for a certain random variable $X$, then $X_n$ converges in distribution to $X$. This implies that
\begin{align*}
\mathbb{P}  \left( \frac{N^{\mathcal{A}}(\cT^\mathcal{B}_k) - k \frac{\mu_{\mathcal{A}}}{\mu_{\mathcal{B}}}}{\sqrt{k}} \in (u,v) \right) & \rightarrow 
\displaystyle\int_{u}^{v} \left[ \displaystyle\int_{\mathbb{R}} \frac{1}{\sqrt{2 \pi \gamma^2}} e^{-\frac{y^2}{2 \gamma^2}} \frac{1}{\sqrt{2 \pi \sigma_{\cA,\cB}^2}}  e^{-\frac{1}{2 \sigma_{\cA,\cB}^2}{\left(z-C_{\cA,\cB}y \right)^2}} dy \right]dz\\
&=  \int_u^v\frac{1}{\sqrt{2 \pi \delta_{\cA,\cB}^2 }} e^{- \frac{1}{2 \delta_{\cA,\cB}^2} z^2} dz
\end{align*} 
with $\delta_{\cA,\cB}^2=C_{\cA,\cB}^2 \gamma^2 + \sigma_{\cA,\cB}^2$>0. We leave the case where at least one of the quantities $\mu_{\cA \cap \cB}$, $\mu_{\cA \backslash \cB}$, $\mu_{\cB \backslash \cA}$, $\mu_{\cA^c \cap \cB^c}$ is $0$ to the reader, which is treated in the same way. In particular, one gets that $\delta_{\cA,\cB}^2>0$ except when $\mu_\cA = 0$ or $\mu_{\cA \backslash \cB} = \mu_{\cB \backslash \cA} = 0$. This establishes the asymptotic normality of $(N^{\cA}(\cT) | N^{\cB}(\cT) = k)$ with an expression of the limiting variance. 
\end{proof}

The proof of Lemma \ref{lem:normality} is based on the following result, whose proof is a direct adaptation of the proof of Lemma \ref{lem:asymp} in the multivariate setting.

\begin{lem}
\label{lem:asymp2}
Fix $a \in \R$, and let $\left( \cB_1, \ldots, \cB_j \right)$ be a partition of $\Z_+$, satisfying, for all $i \in \llbracket 1,j \rrbracket, \, \mu_{\cB_i} > 0$. Assume in addition that at least one of the laws $p_{\cB_1}, \ldots, p_{\cB_j}$ is aperiodic. For $1 \leq i \leq j$ and $c_i \in \R$, define $n_i(c_i) := \lfloor n \mu_{\cB_i} + c_i \sqrt{n} \rfloor$. Then there exists a symmetric positive definite matrix ${\Sigma} := {\Sigma} \left( \cB_1, \ldots, \cB_j \right) \in \mathcal{S}_j\left( \R \right)$ such that the following assertions hold, uniformly for $\left( c_1, \ldots, c_j \right)$ in a compact subset of $\R^j$ satisfying in addition $\sum_{i=1}^j n_i(c_i)=n$:
\begin{itemize}
\item[(i)]
Let $(a_{n})$ be a sequence of integers such that $a_{n}/\sqrt{n} \rightarrow a$. Then, as $n \rightarrow \infty$,
\begin{align*}
\P \left(S_n=a_n, J_n^{\cB_1} = n_1(c_1), \ldots, J_n^{\cB_j} = n_j(c_j) \right) \sim \frac{1}{\left(2 \pi n \right)^{j/2} \sqrt{\det {\bf \Sigma}}} e^{-\frac{1}{2} ^t\mathbf{x} \Sigma^{-1} \mathbf{x}},
\end{align*}
where ${\bf x} = \left( a, c_1, \ldots, c_{j-1} \right)$.
\item[(ii)]
With the same notations, as $n \rightarrow \infty$, we have
\begin{align*}
\P \left(N^{\Z_+}(\cT)=n, N^{\cB_1}(\cT) = n_1(c_1), \ldots, N^{\cB_j}(\cT) = n_j(c_j) \right) \sim \frac{1}{n} \frac{1}{\left(2 \pi n \right)^{j/2} \sqrt{\det {\bf \Sigma}}} e^{-\frac{1}{2} ^t\bf{x} \Sigma^{-1} \bf{x}},
\end{align*}
where, here, ${\bf x} = \left(0, c_1, \ldots, c_{j-1} \right)$.
\end{itemize}
\end{lem}

\begin{rk}
For convenience, as before, we state here the theorem in the aperiodic case. Remark however that the case where none of the laws $p_{\cB_1}, \ldots, p_{\cB_j}$ are aperiodic boils down to the aperiodic case, up to a change of variables.
\end{rk}

\begin{proof}[Proof of Lemma \ref{lem:normality}] Let us fix $y \in \R$.
First, write
\begin{align*}
\mathbb{P} &\left( \left. \frac{N^{\mathcal{A}}(\cT^\mathcal{B}_k) - k \frac{\mu_{\mathcal{A}}}{\mu_{\mathcal{B}}}}{\sqrt{k}} \in (u,v) \right| N^{\Z_+}(\cT^\mathcal{B}_k) = \lfloor k/\mu_{\mathcal{B}} + \sqrt{k} y \rfloor \right) \\
& \qquad= \frac{\mathbb{P} \left( \frac{N^{\mathcal{A}}(\cT) - k \frac{\mu_{\mathcal{A}}}{\mu_{\mathcal{B}}}}{\sqrt{k}} \in (u,v), N^\mathcal{B}(\cT)=k, N^{\Z_+}(\cT) = \lfloor k/\mu_{\mathcal{B}} + \sqrt{k} y \rfloor \right)}{\mathbb{P} \left( N^\mathcal{B}(\cT)=k, N^{\Z_+}(\cT) = \lfloor k/\mu_{\mathcal{B}} + \sqrt{k} y \rfloor \right)}\\
& \qquad \sim C(y)  k^{5/2} \int_u^v  \mathbb{P} \left( N^{\mathcal{A}}(\cT) = \lfloor k \frac{\mu_{\mathcal{A}}}{\mu_{\mathcal{B}}} + \sqrt{k} h \rfloor , N^\mathcal{B}(\cT)=k, N^{\Z_+}(\cT) = \lfloor k/\mu_{\mathcal{B}} + \sqrt{k} y \rfloor \right) dh,
\end{align*}
where the last asymptotic equivalent follows from \eqref{eq:equivalent} with $C(y) = \frac{2 \pi \sigma \gamma_\cB}{\mu_\cB^2} \exp \left( \frac{y^2 \mu_\cB^3}{2 \gamma_\cB^2} \right)$. In order to prove that this quantity has a limit as $k \to \infty$ and compute it, it is enough to prove that the map $g_k$ defined by
\begin{align*}
g_k(h)= k^{5/2} \mathbb{P} \left(  N^{\mathcal{A}}(\cT) = \lfloor k \frac{\mu_{\mathcal{A}}}{\mu_{\mathcal{B}}} + \sqrt{k} h \rfloor , N^\mathcal{B}(\cT)=k, N^{\Z_+}(\cT) = \lfloor k/\mu_{\mathcal{B}} + \sqrt{k} y \rfloor \right)
\end{align*}
converges uniformly on $(u,v)$ to an integrable function on $(u,v)$.

Remark that we can write $g_k(h) = k^{5/2} \sum_{\ell \in \Z_+} q_\ell$, where

\begin{align*}
q_\ell &= \mathbb{P} \left( N^{\mathcal{A} \cap \mathcal{B}}(\cT) = \ell,  N^{\mathcal{A} \backslash \mathcal{B}}(\cT) = \lfloor k \frac{\mu_{\mathcal{A}}}{\mu_{\mathcal{B}}} + \sqrt{k} h \rfloor - \ell , N^{\mathcal{B} \backslash \mathcal{A}}(\cT)=k - \ell, \right. \\
&\qquad \qquad \qquad \qquad  \left. N^{\cA^c \cap \cB^c}(\cT) = \lfloor k/\mu_{\mathcal{B}} + \sqrt{k} y \rfloor - \lfloor k \frac{\mu_\cA}{\mu_\cB} + \sqrt{k}h \rfloor -k+\ell \right).
\end{align*} 
In other words, we sum over all possible values $\ell$ of $N^{\cA \cap \cB}(\cT)$. The idea is that, if $\ell$ is far from its expectation (namely, $k \mu_{\cA \cap \cB} / \mu_\cB$), then $q_\ell$ is small. On the other hand, we control $q_\ell$ by Lemma \ref{lem:asymp2} when $\ell$ is close to its expectation.
More specifically, set
\begin{align*}
I_k(h) := \left\{ \ell \in \Z_+; \left( \ell, \lfloor k \frac{\mu_{\mathcal{A}}}{\mu_{\mathcal{B}}} + \sqrt{k} h \rfloor - \ell, k-\ell, \lfloor \frac{k}{\mu_\cB} + \sqrt{k}y \rfloor - \lfloor k \frac{\mu_\cA}{\mu_\cB} + \sqrt{k}h \rfloor -k+\ell \right) \in \mathbb{C}_k,  \right\},
\end{align*}
where $$\mathbb{C}_k = \left\{ k \frac{\mu_{\cA \cap \cB}}{\mu_\cB}, k \frac{\mu_{\cA \backslash \cB}}{\mu_\cB}, k \frac{\mu_{\cB \backslash \cA}}{\mu_\cB}, k \frac{\mu_{\cA^c \cap \cB^c}}{\mu_\cB} \right\} + \left[-k^{3/5}, k^{3/5}\right]^4.$$

First, remark that, for $\ell \in \Z_+$, $q_\ell \leq \sum_{i=1}^4 \P(N^{\cA_i}(\cT)=\ell \big| N^{\Z_+}(\cT)=\lfloor \frac{k}{\mu_\cB} + \sqrt{k}y \rfloor)$, where $(\cA_1, \cA_2, \cA_3, \cA_4) := (\cA \cap \cB, \cA \backslash \cB, \cB \backslash \cA, \cA^c \cap \cB^c)$.  Therefore,
\begin{align*}
\sum_{\ell \notin I_k(h)} q_\ell &\leq \sum_{i=1}^4 \P \left( \left| N^{\cA_i}(\cT) - \frac{k \mu_{\cA_i}}{\mu_\cB} \right| \geq k^{3/5} \big| N^{\Z_+}(\cT)=\lfloor \frac{k}{\mu_\cB} + \sqrt{k}y \rfloor \right)\\
&= \sum_{i=1}^4 \P \left( |B_i-\E[B_i]| \geq k^{3/5} \right) \left(1+o(1)\right),
\end{align*}
where $B_i \sim Bin (\lfloor k/\mu_\cB \rfloor, \mu_{\cA_i})$.
Thus, using Hoeffding inequality, we get:
\begin{equation}
\label{eq:1.1}
\sum_{\ell \notin I_k(h)} q_\ell = oe(k)
\end{equation}
uniformly in $h \in \R$.

On the other hand, by Lemma \ref{lem:asymp2} (ii), there exists an invertible matrix $\Sigma \in \mathcal{S}_{4}(\R)$ and a constant $C_1 > 0$ such that, uniformly for $h \in \R$,
\begin{equation}
\label{eq:1.2}
\sum_{\ell \in I_k(h)} q_\ell \sim \sum_{\ell \in I_k(h)} C_1 k^{-3} e^{-\frac{1}{2} ^tx_\ell \Sigma^{-1} x_\ell}
\end{equation}
for $x_\ell := (({\ell-k\frac{\mu_{\cA \cap \cB}}{\mu_\cB}})/{\sqrt{k}},y,h,0)$.

By Equations \eqref{eq:1.1} and \eqref{eq:1.2}, by summing over all $\ell \in \Z_+$, we get that, as $k \rightarrow \infty$, uniformly in $h \in \R$, $g_k(h) \rightarrow C_2 \exp \left( -B_{\cA,\cB} \left( h-C_{\cA,\cB}y \right)^2 \right)$ for a certain $C_2$ depending on $\cA, \cB$, and some constants $B_{\cA,\cB},C_{\cA,\cB}$ depending on $\cA$ and $\cB$. Since this limiting function is integrable, by uniform convergence, for any $u,v \in \R \cup \{ +\infty, -\infty \}$, 
\begin{align*}
 \mathbb{P} &\left( \left. \frac{N^{\mathcal{A}}(\cT^\mathcal{B}_k) - k \frac{\mu_{\mathcal{A}}}{\mu_{\mathcal{B}}}}{\sqrt{k}} \in (u,v) \right| N^{\Z_+}(\cT^\mathcal{B}_k) = \lfloor \frac{k}{\mu_{\cB}} + \sqrt{k} y \rfloor \right)  \underset{k \rightarrow \infty}{\longrightarrow} \tilde{C}(y) \int_u^v e^{ - B_{\cA,\cB} \left( h- C_{\cA,\cB} y\right)^2} dh,
\end{align*}
where $ \tilde{C}(y)$ is a constant only depending on $y$ (and $\cA, \cB$).
By taking $u=-\infty$ and $v=+\infty$, one sees  that $\tilde{C}(y)$ does not depend on $y$. Hence, there exists $\sigma_{\cA,\cB}^2>0$ such that, for any $y \in \R$, $\tilde{C}(y)={1}/{\sqrt{2\pi \sigma_{\cA,\cB}^2}}$. Furthermore, by taking again $u=-\infty$ and $v=+\infty$, the value of the right hand side is  $1$, which tells us that $B_{\cA,\cB}=\frac{1}{2\sigma_{\cA,\cB}^2}$. Finally, we conclude that for every $y \in \R$ and $u<v$:

\begin{align*}
 \mathbb{P} &\left( \left. \frac{N^{\mathcal{A}}(\cT^\mathcal{B}_k) - k \frac{\mu_{\mathcal{A}}}{\mu_{\mathcal{B}}}}{\sqrt{k}} \in (u,v) \right| N^{\Z_+}(\cT^\mathcal{B}_k) = \lfloor \frac{k}{\mu_{\cB}}  + \sqrt{k} y \rfloor \right) \underset{k \rightarrow \infty}{\longrightarrow} 
 \frac{1}{\sqrt{2\pi \sigma_{\cA,\cB}^2}} \int_u^v e^{-\frac{1}{2 \sigma_{\cA,\cB}^2}{\left( h-C_{\cA,\cB}y\right)^2}} dh
\end{align*}
which completes the proof of Lemma \ref{lem:normality}.
\end{proof}

Finally, we briefly present the proof of Theorem \ref{thm:normality} (iii), which is based again on Lemma \ref{lem:asymp2} (ii).

Let us consider the tree $\cT_k^\cB$ for a certain $\cB \subset \Z_+$. Let $\cA_1, \ldots\cA_j \subset \Z_+$. It induces a partition of $\Z_+$ made of the set $E := \left\{ \cap_{i=1}^{j+1} \mathcal{C}_i, \mathcal{C}_i \in \left\{ \cA_i, \cA_i^c \right\}, \mathcal{C}_{j+1} \in \left\{ \cB, \cB^c \right\} \right\} \backslash \{ \emptyset \}$. Let $(u_i,v_i)_{1 \leq i \leq j}$ be real numbers with $u_i<v_i$ for every $1 \leq i \leq j$. Then
\begin{align*}
\P &\left( \frac{N^{\cA_1}(\cT_k^\cB)-k\frac{\mu_{\cA_1}}{\mu_\cB}}{\sqrt{k}} \in (u_1,v_1), \ldots, \frac{N^{\cA_j}(\cT_k^\cB)-k\frac{\mu_{\cA_j}}{\mu_\cB}}{\sqrt{k}} \in (u_j,v_j) \right) \\
&= \sum_{n \in \Z_+} \frac{\P \left( N^{\Z_+}(\cT)=n \right)}{\P \left( N^\cB(\cT)=k \right)} \\
& \qquad \qquad  \times \P \left( \frac{N^{\cA_1}(\cT_n^{\Z_+})-k\frac{\mu_{\cA_1}}{\mu_\cB}}{\sqrt{k}} \in (u_1,v_1), \ldots, \frac{N^{\cA_j}(\cT_n^{\Z_+})-k\frac{\mu_{\cA_j}}{\mu_\cB}}{\sqrt{k}} \in (u_j,v_j), N^\cB(\cT_n^{\Z_+})=k \right)\\
&= \sum_{n \in \Z_+} \frac{\P \left( N^{\Z_+}(\cT)=n \right)}{\P \left( N^\cB(\cT)=k \right)} \sum_{\left( x_{\mathcal{H}} \right)_{\mathcal{H} \in E} \in I_n}  \P \left( \underset{\mathcal{H} \in E}{\cap} N^{\mathcal{H}}\left( \cT_n^{\Z_+} \right) = x_{\mathcal{H}} \right)
\end{align*}
for some finite set $I_n \in \Z_+^{|E|}$.
We can now rewrite this probability in terms of random walks and use Lemma \ref{lem:asymp2} (ii) in order to get the asymptotic normality of the quantity 
\begin{align*}
\P \left( \frac{N^{\cA_1}(\cT_k^\cB)-k\frac{\mu_{\cA_1}}{\mu_\cB}}{\sqrt{n}} \in (u_1,v_1), \ldots, \frac{N^{\cA_j}(\cT_k^\cB)-k\frac{\mu_{\cA_j}}{\mu_\cB}}{\sqrt{n}} \in (u_j,v_j) \right).
\end{align*}

\begin{ex}
\label{ex:example}
In explicit cases, it is possible to carry out the calculations in the proof of Theorem \ref{thm:normality} to compute the value of $\delta_{\cA,\cB}$ and of the covariances. We give several examples:
\begin{enumerate}
\item[--] In the case $\cB=\Z_{+}$ and $ \mathcal{A}= \{r\}$ with $r \geq 1$ (which was treated by \cite{Jan16}), one has $\delta_{\cA,\cB}^{2}= \mu_{r}(1-\mu_{r})-(r-1)^{2} \mu_{r}^{2}/\sigma^{2}$ and the covariance between the limiting Gaussian random variables for $\cA_{1}= \{r\}$ and $\cA_{2}= \{s\}$ is $-\mu_{r} \mu_{s}-(r-1)(s-1) \mu_{r} \mu_{s}/\sigma^{2}$.
\item[--] In the case $\cB= \{a\}$ for some $a \in \Z_+$ and $\mathcal{A}= \{r\}$, one has $\delta_{\cA,\cB}^{2}=\frac{\mu_{r}}{\mu_{a}} ( 1+\frac{\mu_{r}}{\mu_{a}} ) -  \frac{(r-a)^{2} \mu_r^2}{\mu_a \sigma^{2}}$ and the covariance between the limiting Gaussian random variables for $\cA_{1}= \{r\}$ and $\cA_{2}= \{s\}$ is $ \frac{\mu_r \mu_s}{\mu_a^2} \left( 1- (r-a) (s-a) \frac{\mu_{a}}{\sigma^{2}}\right)$.
\item[--] In particular, in the case $\cB= \{0\}$ (this corresponds to conditioning by a fixed number of leaves, and is useful in the study of dissections \cite{Kor14}) and $\mathcal{A}= \{r\}$, one has $\delta_{\cA,\cB}^{2}=\frac{\mu_{r}}{\mu_{0}} ( 1+\frac{\mu_{r}}{\mu_{0}} ) -  \frac{r^{2} \mu_r^2}{\mu_0 \sigma^{2}}$ and the covariance between the limiting Gaussian random variables for $\cA_{1}= \{r\}$ and $\cA_{2}= \{s\}$ is $ \frac{\mu_r \mu_s}{\mu_0^2}\left( 1- r s  \frac{\mu_{0} }{\sigma^{2}}\right)$.
\item[--] In the case $\cB= \{0\}$ and $\cA=\mathbb{Z}_{+}$, by \eqref{eq:NZ}, $\delta_{ \mathbb{Z}_{+}, \{0\} }^{2}= \frac{1-\mu_{0}}{\mu_{0}^{2}}- \frac{1}{\mu_{0} \sigma^{2}}$.
\end{enumerate}
\end{ex}

\begin{rk}Using the same arguments as in the end of this Section, it is possible to show that convergences of the exploration processes in Theorem \ref{thm:processus} hold jointly for $\cA_1, ..., \cA_k \subset \Z_{+}$ (with correlated Brownian motions), and to extend the results with $\cT_n$ replaced with $\cT^{\cB}_n$.
\end{rk}

\section{Several extensions}
\label{ss:ext}

We now present some possible extensions of Theorems \ref{thm:processus} and \ref{thm:normality} for other types of offspring distributions. A natural one is the extension of these results to distributions $\mu$ that are said to be in the domain of attraction of a stable law. We first properly define this notion, before explaining how the two abovementioned theorems can be generalized in this broader framework. The second extension that we present is the case of subcritical non-generic laws, where the offspring distribution is not critical anymore. In this case, we asymptotically observe in the random tree $\cT_n$ a condensation phenomenon, where one vertex has macroscopic degree. See e.g. \cite[Example $19.33$]{Jan12} for more context.

\subsection{Stable offspring distributions.} 

Let us first provide some background. 
We say that a function $L: \R_+^* \rightarrow \R_+^*$ is slowly varying if, for any $c>0$, $L(cx)/L(x) \rightarrow 1$ as $x \rightarrow \infty$. For $\alpha \in (1,2]$, we say that a critical distribution $\mu$ belongs to the domain of attraction of an $\alpha$-stable law if either $\mu$ has finite variance (in which case $\alpha=2$) or there exists a slowly varying function $L$ such that
\begin{equation}
\label{eq:L}
Var \left( X \mathds{1}_{X \leq x} \right) \underset{x \rightarrow \infty}{\sim} x^{2-\alpha} L(x),
\end{equation}
where $X$ is a random variable of law $\mu$.
In this case, for any sequence $(D_n)_{n \geq 1}$ of positive numbers satisfying
\begin{equation}
\label{eq:Dn}
\frac{n L(D_n)}{D_n^\alpha} \underset{n \rightarrow \infty}{\rightarrow} \frac{\alpha(\alpha-1)}{\Gamma(3-\alpha)},
\end{equation}
we have the following joint convergence:

\begin{thm}
\label{thm:stableprocessus}
Let $\alpha \in (1,2]$ and $\mu$ a critical distribution with infinite variance in the domain of attraction of an $\alpha$-stable law. Let $(D_n)_{n \geq 1}$ be a sequence satisfying \eqref{eq:Dn}. Then, there exists two nondegenerate random processes $X^{(\alpha)}, H^{(\alpha)}$, depending only on $\alpha$, such that the following convergences hold jointly:

\begin{itemize}
\item[(i)] We have
\begin{align*}
\left( \frac{W_{nt}(\mathcal{T}_n)}{D_n} , \frac{K^{\cA}_{nt}(\cT_n)-nt \mu_\cA}{\sqrt{n}} \right) _{0 \leq t \leq 1}  \quad \mathop{\longrightarrow}^{(d)}_{n \rightarrow \infty} \quad 
\left( X^{(\alpha)}_t, \sqrt{\mu_\cA(1-\mu_\cA)} B_t \right) _{0 \leq t \leq 1}.
\end{align*}

\item[(ii)] The following convergence holds in distribution, jointly with that of (i):
\begin{align*}
\left( \frac{D_n}{n} C_{2nt}(\mathcal{T}_n),  \frac{N^{\cA}_{2nt}(\mathcal{T}_n) - nt \mu_\cA}{\sqrt{n}}  \right)_{0 \leq t \leq 1} 
 \quad \mathop{\longrightarrow}^{(d)}_{n \rightarrow \infty} \quad 
\left( H^{(\alpha)}_t, \sqrt{\mu_\cA(1-\mu_\cA)} {B}_t \right) _{0 \leq t \leq 1}.
\end{align*}
\end{itemize}
Here, $B$ denotes is a standard Brownian motion independent of $(X^{(\alpha)},H^{(\alpha)})$.
\end{thm}

The processes $X^{(\alpha)}, H^{(\alpha)}$ only depend on $\alpha$, and are the continuous-time analogues of respectively the Lukasiewicz path and the contour function of the so-called $\alpha$-stable tree (see Fig. \ref{fig:stabletree} for a picture, and \cite{DLG02} for more details). This stable tree is a random compact metric space introduced by Duquesne and Le Gall \cite{DLG02}, known to be the scaling limit of the sequence of size-conditioned $\mu$-Galton-Watson trees $(\cT_n)$, when $\mu$ is in the domain of attraction of an $\alpha$-stable law. Notably, when $\alpha=2$, $X^{(2)} = H^{(2)} = \e$.

\begin{figure}[ht!]
\center
\caption{An approximation of the $\alpha$-stable tree and the processes $X^{(\alpha)}$ and $H^{(\alpha)}$, for $\alpha=1.6$.}
\label{fig:stabletree}
\includegraphics[scale=.8]{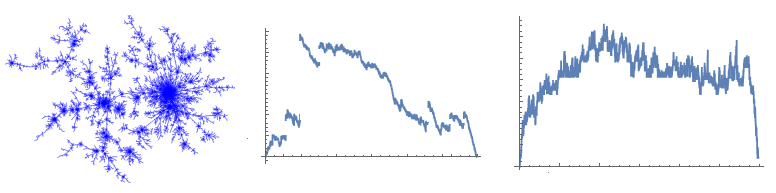}
\end{figure}

Note that, setting $\sigma^2=\infty$ in the definition of $\gamma_\cA$ given in Theorem \ref{thm:processus}, we obtain exactly $\gamma_\cA = \sqrt{\mu_\cA(1-\mu_\cA)}$, so that Theorem \ref{thm:stableprocessus} is indeed the natural generalization of the finite variance case. An interesting remark, in the infinite variance case, is that the two marginals of the limiting processes are independent. 

On the other hand, the results of Theorem \ref{thm:normality} still hold in this case:

\begin{thm}
\label{thm:stablenormality}
Let $\mu$ be a critical offspring distribution with infinite variance in the domain of attraction of a stable law, and let $\cA, \cB $ be subsets of $\mathbb{Z}_{+}$ such that $\mu_\cB>0$. For $n \geq 1$, let $\mathcal{T}_n^{\cB}$ be a $\mu$-GW tree conditioned to have $n$ $\cB$-vertices. Then:
\begin{enumerate}
\item[(i)]  as $n \rightarrow \infty$, $\frac{1}{n}\mathbb{E}(N^{\cA}(\mathcal{T}_n^\cB)) \rightarrow \frac{\mu_{\cA}}{\mu_{\cB}}$;
\item[(ii)] there exists $\delta_{\cA,\cB} \geq 0$ such that the convergence
\begin{equation*}
\frac{N^{\cA}(\mathcal{T}_n^{\cB})-n \frac{\mu_{\cA}}{\mu_{\cB}}}{\sqrt{n}}  \quad \mathop{\longrightarrow}^{(d)}_{n \rightarrow \infty} \quad  \mathcal{N}(0,\delta_{\cA,\cB}^2)
\end{equation*}
holds in distribution, where $\mathcal{N}(0,\delta_{\cA,\cB}^2)$ is a centered Gaussian random variable with variance $\delta_{\cA,\cB}^2$. In addition, $\delta_{\cA,\cB}=0$ if and only if $\mu_\cA=0$ or $\mu_{\cA \backslash \cB} = \mu_{\cB \backslash \cA} = 0$.
\item[(iii)] the convergences \eqref{eq:normalite} hold jointly for $ \mathcal{A} \subset \Z_{+}$, in the sense that for every $j \geq 1$ and $\cA_1, \cdots, \cA_j \subset \Z_{+}$, $ (({N^{\cA_i}(\mathcal{T}_n^{\cB})-n \frac{\mu_{\cA_i}}{\mu_{\cB}}})/{\sqrt{n}})_{1 \leq i \leq j} $ converges in distribution to a Gaussian vector.
\end{enumerate}
\end{thm}

These two generalizations can be obtained by slightly adapting the proofs of Theorems \ref{thm:processus} and Theorem \ref{thm:normality}. Let us only explain the most important changes in these proofs, which consist in generalizing Theorems \ref{thm:CW} and \ref{thm:multivariatellt} in the stable framework:

\begin{thm}[Duquesne \& Le Gall \cite{DLG02}]
\label{thm:stableduquesne}
Let $\alpha \in (1,2]$, and let $\mu$ a critical distribution in the domain of attraction of an $\alpha$-stable law. Let $(D_n)_{n \geq 1}$ be a sequence satisfying \eqref{eq:Dn}. Then, the following convergence holds jointly in distribution:
\begin{align*}
\left( \frac{D_n}{n} C_{2nt}(\cT_n) , \frac{1}{D_n} W_{nt}(\cT_n) \right)_{0 \leq t \leq 1} \underset{n \rightarrow \infty}{\overset{(d)}{\rightarrow}} \left( H^{(\alpha)}_t, X^{(\alpha)}_t \right)_{0 \leq t \leq 1}.
\end{align*}
\end{thm}

The other ingredient is a multivariate local limit theorem in the stable case. When the first coordinate of a random vector is in the domain of attraction of a stable law and has infinite variance, while all others coordinates have finite variance, the random vector satisfies a local limit theorem. In addition, the first coordinate of the limiting object is independent of all others, which themselves are distributed as a Gaussian vector:

\begin{thm}[Resnick \& Greenwood \cite{RG79}, Hahn \& Klass \cite{HK80}, Doney \cite{Don91}]
\label{thm:stablemultivariatellt}
Let $\alpha \in (1,2]$. Let $j \geq 1$ and $(\mathbf{Y}_i)_{i \geq 1} := ((Y_i^{(1)}, \ldots, Y_i^{(j)}))_{i \geq 1}$ be i.i.d. random variables in $\Z^j$, such that $Y_1^{(1)}$ is in the domain of attraction of an $\alpha$-stable law $\mu$ and has infinite variance, and that the covariance matrix $\Sigma$ of the vector $(Y_1^{(2)}, \ldots, Y_1^{(j)})$ is symmetric positive definite. Assume in addition that $\mathbf{Y}_1$ is aperiodic, and denote by $M^{(k)}$ the mean of $Y_1^{(k)}$, for $1 \leq k \leq j$. Finally, define for $n \geq 1$
\begin{align*}
\mathbf{T_n} = \sum_{i=1}^n \left( \frac{Y_i^{(1)}-M^{(1)}}{D_n}, \frac{Y_i^{(2)}-M^{(2)}}{\sqrt{n}}, \ldots, \frac{Y_i^{(j)}-M^{(j)}}{\sqrt{n}} \right)
\end{align*}
Then, as $n \rightarrow \infty$, uniformly for $\mathbf{x} := (x^{(1)}, \ldots, x^{(j)})$ in a compact subset of $\R^j$ satisfying $\P \left( \mathbf{T_n} = \mathbf{x} \right) > 0$,
\begin{align*}
\P \left( \mathbf{T_n} = \mathbf{x} \right) \sim  \frac{g \left(x^{(1)} \right)}{D_n} \times \frac{1}{(2\pi n)^{(j-1)/2} \sqrt{\det \Sigma}} e^{-\frac{1}{2} ^t \mathbf{\tilde{x}} \Sigma^{-1} \mathbf{\tilde{x}}},
\end{align*}
where $g$ is the density of $\mu$ and $\tilde{x} := (x^{(2)}, \ldots, x^{(j)})$.
\end{thm}

Let us explain how we obtain this result, by combining the results of \cite{Don91}, \cite{HK80} and \cite{RG79}. We first focus on the case $j=2$. When $\alpha \in (1,2)$, \cite[Theorem $3$]{RG79} states that the convergences of the two marginals $\sum_{i=1}^n (Y_i^{(1)}-M^{(1)})/D_n$ and $\sum_{i=1}^n (Y_i^{(2)}-M^{(2)})/\sqrt{n}$ hold, and that obtaining these two convergences separately is enough to get Theorem \ref{thm:stablemultivariatellt}. The same theorem states in addition that the two limiting marginals are independent.

On the other hand, when $j=2$, $\alpha=2$ and $\mu$ has infinite variance, Theorem $3$ in \cite{HK80} shows that $\textbf{T}_n$ converges in distribution to a bivariate normal variable, and that the first coordinate of the limiting distribution is independent of the second (the constant $\gamma_n$ that appears in the statement of \cite[Theorem $3$]{HK80} can be proved to be $0$, so that the renormalization matrix $A_n$ appearing in this theorem is diagonal). This, coupled with \cite[Theorem $1$]{Don91} (which, roughly speaking, states that a bivariate central limit theorem implies a local limit theorem), implies Theorem \ref{thm:stablemultivariatellt} in the case $\alpha=2$, $j=2$. 
Although these results are only stated for $j=2$ (with the exception of \cite[Theorem $3$]{HK80}, which is generalized in \cite[Theorem $5$]{HK80}), they still hold for $j \geq 3$ with mild motifications.

The proof of Theorem \ref{thm:stableprocessus} follows the proof of Theorem \ref{thm:processus} in the finite variance case, applying Theorem \ref{thm:stablemultivariatellt} to the random vector $(S_1, J_1^{\cA})$. In order to generalize the results of Theorem \ref{thm:normality} to the infinite variance case, we apply Theorem \ref{thm:stablemultivariatellt} to the vector $(S_n, J_n^{\cB \cap \cA}, J_n^{\cB \backslash \cA}, J_n^{\cA \backslash \cB}, J_n^{\cA^c \cap \cB^c})$.

\paragraph*{Convergence of $\cT_n^{\cA}$ to the stable tree}

We finish the study of the stable case by proving the convergence of the conditioned trees $(\cT_n^\cA)$, properly renormalized, to the stable tree, for any $\cA \subset \Z_+$ satisfying $\mu_\cA>0$. More precisely, the multivariate theorem \ref{thm:multivariatellt}, along with Proposition \ref{prop:joint}, allows to obtain the following asymptotics, which generalizes \cite[Theorem $8.1$ (i)]{Kor12}:

\begin{prop}
\label{prop:general8.1}
Let $\alpha \in (1,2]$, and let $\mu$ be in the domain of attraction of an $\alpha$-stable law with infinite variance. Let $\cA \subset \Z_+$ be such that $\mu_\cA > 0, \mu_{\cA^c} >0$ and $\cT$ be a $\mu$-GW tree. Then, there exists a constant $C$ depending only on $\mu$ and $\cA$ such that the following holds as $n \rightarrow \infty$, for the values of $n$ such that $\P \left( N^\cA(\cT) = n \right) > 0$:
\begin{align*}
\P \left( N^\cA(\cT) = n \right) = \sum_{k \geq 0} \P \left( N^{\Z_+}(\cT) = k, N^\cA(\cT) = n \right) \sim \frac{C}{L(n) n^{1+1/\alpha}},
\end{align*}
where $L$ verifies \eqref{eq:L}.
\end{prop}

Note that our bivariate approach allows to prove it for all $\cA \subset \Z_+$, while \cite[Theorem $8.1$ (i)]{Kor12} holds only when $\cA$ or $\Z_+ \backslash \cA$ is finite. An immediate corollary of Proposition \ref{prop:general8.1} is the joint convergence of the contour function and the Lukasiewicz path of the conditioned tree $\cT_n^\cA$:

\begin{cor}
\label{cor:8.1.II}
Restricting ourselves to the values of $n$ such that $\P(N^\cA(\cT) = n)>0$,
\begin{align*}
\left( \frac{D_{N(\cT_n^\cA)}}{N(\cT_n^\cA)} C_{2N(\cT_n^\cA)t}(\cT^\cA_n) , \frac{1}{D_{N(\cT_n^\cA)}} W_{N(\cT_n^\cA)t}(\cT_n^\cA) \right)_{0 \leq t \leq 1} \underset{n \rightarrow \infty}{\overset{(d)}{\rightarrow}} \left( H^{(\alpha)}_t, X^{(\alpha)}_t \right)_{0 \leq t \leq 1}.
\end{align*}
\end{cor}

The proof of this corollary follows exactly the proof of \cite[Theorem $8.1$ (II)]{Kor12}. In particular, this convergence implies the convergence in distribution of the tree $\cT_n^\cA$, viewed as a metric space for the graph distance and properly renormalized, towards the $\alpha$-stable tree for the Gromov-Hausdorff distance (see e.g. \cite[Section $2$]{LG05} for details).

\subsection{Subcritical non-generic offspring distributions.} We now focus on the case where $\mu$ is subcritical (that is with mean strictly less than $1$) and  $\mu_k \sim c k^{-\beta}$ as $k \to \infty$, with fixed  $c > 0$ and $\beta > 2$, and $\cB=\mathbb{Z}_+$. This is an interesting case, as a condensation phenomenon occurs (see \cite{Jan12,Kor15}): a unique vertex with macroscopic degree comparable to the total size of the tree emerges. Then the following asymptotic normality holds.

\begin{thm}
\label{euro}
Assume that   $\mu$ is an offspring distribution such that $\mu_k \sim c k^{-\beta}$ as $k \to \infty$, with fixed  $c > 0$ and $\beta > 2$, and denote by $\mathcal{T}_n$ a $\mu$-GW tree conditioned to have $n$ vertices.
Let $k \geq 1$ and $\cA_1, \cA_2, \ldots, \cA_k \subset \mathbb{Z}_+$ be finite. Then we have the joint convergence in distribution
 \[\left(\frac{N^{\cA_1}(\mathcal{T}_n)-n \mu_{\cA_1}}{\sqrt{n}}, \ldots, \frac{N^{\cA_k}(\mathcal{T}_n)-n \mu_{\cA_k}}{\sqrt{n}}\right) \longrightarrow (Z_{\cA_1}, \ldots , Z_{\cA_k}),\] where $Z_{\cA_i} \sim \mathcal{N}(0,\mu_{\cA_i}(1-\mu_{\cA_i}))$ and for $i \neq j$:
\[
Cov(Z_{\cA_i}, Z_{\cA_j})
= \mu_{\cA_i \cap \cA_j} - \mu_{\cA_i} \mu_{\cA_j}.\]
\end{thm}

\begin{proof}
By \cite[Theorem 1]{AL11} (see \cite[Sec.~2.1]{Kor15} for its use in this context) or \cite[Theorem 19.34]{Jan12}  after removing the largest outdegree in $\cT_{n}$, the other outdegrees are asymptotically i.i.d. with distribution $\mu$. Therefore, for every $M \geq 1$, the law of the vector $\left(N^{\{1\}}(\mathcal{T}_n), \ldots, N^{\{M\}}(\mathcal{T}_n)\right)$ is asymptotically multinomial with parameters $\left( n, \mu_1, \ldots , \mu_M \right)$. The result follows.
\end{proof}

\paragraph*{Conjecture.} We have seen that the conclusions of Theorem \ref{euro} hold for $\mu$ with infinite variance in the domain of attraction of a stable law 
and for $\mu$ a subcritical power law. We believe that these conclusions should hold for any $\mu$ critical with infinite variance, as well as for $\mu$ subcritical with no exponential moment. In particular, we should get, for any $\cA \subset \Z_+$, $({N^\cA (\cT_n) - n \mu_\cA})/{\sqrt{n}} \overset{d}{\rightarrow} \mathcal{N} \left( 0, \mu_\cA \left( 1-\mu_\cA \right) \right)$. However, in the general case, nothing is known about the scaling limits of such GW trees (see \cite{Jan12} for detailed arguments and counterexamples) and no general local limit theorem exists, which prevents us from directly generalizing our methods.

\bibliographystyle{abbrv}
\bibliography{bibli}
\end{document}